\documentclass[12,reqno]{amsart}
\usepackage{amsfonts}
\usepackage{amsmath} 
\usepackage{mathrsfs}
\usepackage{amssymb,mathtools}
\usepackage{color,amscd,array,graphicx,mathdots,bm,amsfonts,epsfig,amsmath}
\usepackage{amsthm}
\usepackage{enumerate}  
\usepackage{geometry}
\usepackage{comment}
\usepackage{caption}
\usepackage{subfigure}
\usepackage[T1]{fontenc}
\usepackage{blindtext}
\usepackage{cite}
\usepackage[toc,page]{appendix}
\usepackage{epstopdf} 
\usepackage{textgreek}
\newtheorem{theorem}{Theorem}[section]

\newtheorem{lemma}[theorem]{Lemma}
\newtheorem{example}[theorem]{Example}

\newtheorem{remark}[theorem]{Remark}

\def\ba{\begin{eqnarray}}
	\def\ea{\end{eqnarray}}
\def\bas{\begin{eqnarray*}}
	\def\eas{\end{eqnarray*}}
\usepackage{graphicx,color,float} 

\usepackage[colorlinks=true,citecolor=blue]{hyperref}
\newtheorem{assumption}{Assumption}[section]

\usepackage{ulem}

\def\al{{\alpha}}\def\be{{\beta}}\def\de{{\delta}}
\def\ga{{\gamma}}
\def\la{{\lambda}}\def\si{{\sigma}}

\def\De{{\Delta}}\def\Ga{{\Gamma}}
\def\La{{\Lambda}}\def\Om{{\Omega}}

\def\EE{\mathbb E}\def\PP{\mathbb P}
\def\RR{\mathbb R}
\def\PP{\mathbb P}
\def\cB{{\mathcal B}}
\def\cF{{\mathcal F}}

\def\cO{{\mathcal O}}

\def\1{\mathbf 1}

\def\diag{\hbox{Diag}}
\def\d{\hbox{d}}

\def\<{\left<}\def\>{\right>}

\def\({\left(}\def\){\right)}

\numberwithin{equation}{section}

\begin{document}

\title{Long time numerical stability of implicit schemes for stochastic heat equations}
	

\author{Xiaochen Yang} 
\author{Yaozhong Hu}

\thanks{Y. Hu is supported by an NSERC Discovery grant and a centennial  fund from University of Alberta.  }

\address{School of Mathematical and Statistical Sciences, Harbin Institute of Technology, China.}
\email{\href{mailto:Yangxc\_math@163.com}{Yangxc\_math@163.com}}

\address{Department of Mathematical and Statistical Sciences, University of Alberta, Edmonton, AB, T6G 2G1, Canada.}
\email{\href{mailto:yaozhong@ualberta.ca}{yaozhong@ualberta.ca}}

\date{\today}

\begin{abstract}
	This paper studies  the long time stability  of  both  stochastic heat equations 
	on a bounded domain  driven by a correlated noise and their approximations. 
	It is   popular for researchers to prove the intermittency of the solution 
	which means that the moments of solution to  stochastic heat equation 
	usually  grow exponentially to infinite  and this   hints  that the 
	solution to stochastic heat equation is generally not stable in long time.
	However, quite surprisingly in this paper 
	we show that when the domain is bounded and when the noise 
	is  not singular in spatial variables, the system can be  long time stable
	and we also prove that we can approximate the solution by its finite dimensional
	spectral approximation  which  is also long time stable. 
	The idea is to use  eigenfunction expansion of the Laplacian on   bounded domain.   
	We also present numerical experiments which are consistent with our theoretical results.
\end{abstract}


\maketitle

\normalsize

\section{Introduction} 
Let   $\mathcal{O} $ be  a bounded domain with smooth boundary $\partial \mathcal{O}$ in  
$d$-dimensional Euclidean space $\RR^d$
and let   $\left(\{ \dot W(t,x) = \frac{\partial^{d+1}}{\partial t \partial x_1 \partial x_2 \cdots \partial x_d}W(t,x)
\,, t\ge 0\,, x\in  \cO\right\}$  be a  mean zero
Gaussian noise on some probability space  $(\Om, \cF, \PP)$ with covariance 
\begin{equation}
	\EE(\dot W(t,x) \dot W(s,y))=\delta(t-s) q(x,y)\,,\quad t, s\ge 0\,,\quad x,y\in \cO\,, 
\end{equation}
where $\EE$ denotes the expectation on the probability space  $(\Om, \cF, \PP)$, 
$\delta$ is the Dirac delta function, and $q(x,y)$ is some positive and positive definite 
function. This means that  our noise is white in time variable and colored in spatial variables.
We shall study the  stability of both the solution and its    approximations  of the following 
stochastic heat equation   
\begin{equation}\label{Heatequation}
	\frac{	\partial u}{\partial  t} (t,x) = \Delta u(t,x) - \beta_0 u(t,x) +\beta_1 u(t,x) \dot W(t,x),  \quad t\ge 0\,,  ~~~~~x\in \mathcal{O}
\end{equation}
with  the  Dirichlet boundary condition that $u(t,x)=0$ for all $x\in \partial \cO$ and 
with a given  initial value $u^0(x), x\in \cO$.    Here $\beta_0$ and $\beta_1$ are certain given real numbers; the symbol   $\Delta =\sum_{i=1}^d \frac{\partial ^2}{\partial x_i ^2}$
denotes  the Laplacian;    and the above product between the random field
$u(t,x)$ and the generalized random field $\dot W(t,x)$  
will be interpreted in the It\^o-Skorohod  sense (we refer to \cite{HuHuang15,Hu19} and references therein for the definition, existence, uniqueness
and properties of the solution).     

Due to its connection to the famous Kardar-Parisi-Zhang  (KPZ)   equation
and to the Anderson localization, 
the above  equation, occasionally in the name of parabolic  Anderson model,   has been studied by many researchers and a lot of results have been obtained for  deep understanding of the solution.  Relevant to our
work, we just mention that if $\cO$ is the entire Eulcidean space $\RR^d$, then
for many  known noises, it is well-known that 
the solution exists uniquely and usually the moments
will grow exponentially to infinity when  time $t$ tends to infinity,
see   for instance \cite{HuHuang15,HuHuang17,Hu19}  and the  references therein
for further results along this direction. 

The intermittency  results  obtained (e.g. \cite{HuHuang15,HuHuang17,Hu19}) tell us that for the noises studied in the mentioned  works,   when domain is   the whole 
Euclidean space $\RR^d$ the 
solutions to \eqref{Heatequation}  is not stable. In fact, the moments of the solution will go to infinite exponentially fast (e.g. \cite{HuHuang15,HuHuang17,Hu19}  and references therein). Probably because of this reason there has been no research work concerning  the stability of equation 
\eqref{Heatequation}.   Our goal is to demonstrate that under some 
different condition on the noise structure (e.g.  on the covariance $q(x,y)$),  and when the domain $\cO$ is bounded the solution can be  
stable.  In such cases we also construct  implicit  spectral 
approximations   for the 
above equation which are  also stable.

It is natural for researchers  to be still puzzled by the above instability and stability for stochastic 
heat equation. It is natural for researchers to think that the solution of an SHE  will not be stable due to the enormous  studies on intermittency
and wonder why it is possible at all that the solution can be stable.  
To better understand this unstable and stable properties for the equation
\eqref{Heatequation} let us consider a one dimensional analogue 
(geometric Brownian motion) 
of the (infinite dimensional) equation \eqref{Heatequation},  which is given by 
\begin{equation}
	d X(t)=\mu X(t) dt +\si X(t) dB(t)\,,\quad X(0)=x\,,
	\label{e.1.3} 
\end{equation}  
where $\mu$ and $\si$ are constants, $B(t)$ is a Brownian motion,
and $dB(t)$ is It\^o differential. The solution is explicitly given by $X(t)=x \exp \left[\si B(t)+(\mu-\frac{1}{2} \si^2  ) t
\right]$.   Thus,
\[
\|X(t)\|_p^p=\EE \left[ |X(t)|^p\right] =x^p \exp\left[ (\mu +\frac{\si^2 (p-1)}{2} )pt \right]\,. 
\]
Thus, we see rather easily that 
\[
\hbox{Equation \eqref{e.1.3} is $L^p$ stable if and only if }
\quad \mu +\frac{\si^2 (p-1)}{2}<0\,.
\]
In particular, when $p=2$, 
\[
\hbox{Equation \eqref{e.1.3} is mean square stable if and only if }
\quad \mu +\frac{\si^2  }{2}<0\,.
\]
From the above condition we see that  we must require     $\mu<-\frac{\si^2  }{2}$ for the mean square stability
of the solution.  We also refer to \cite{LiHu23} for the mean square stability of
\eqref{e.1.3} for other noise, e.g., fractional Brownian motion and for nonautonomous system.

The above fact motivates us to find conditions so that \eqref{Heatequation} is stable and to study the numerical stability of some approximation  schemes. 
Our first main result     is given in Theorem \ref{exactstability},
which guarantee   the long time mean-square stability of the solution to   Equation \eqref{Heatequation}  under the condition  (the notations $\la_1$ and $\kappa$ below are given in Theorem \ref{exactstability})
\[
2(\la_1+\be_0) -\be_1^2\kappa>0.
\]
This condition is more general and easier to verify than that in 
\cite{Baojianhai2011} (see Remark 2.2 for more details). 

About the  eigenfunction expansion, let us mention that there are two operators associated with 
equation \eqref{Heatequation}. One is the eigenfunctions associated with the Dirichlet Laplace operator $\De$.
Another one is the  operator $Q$ associated with the covariance of the Gaussian noise,
defined by $Qf(x)=\int_{\cO} q(x,y) f(y) dy$  for  nice functions  $f:\cO\to \RR$ 
(which we shall not use and hence will not go to details).    
It is worth to mention that we use  the eigenfunction basis of the Dirichlet 
Laplacian instead of the operator $Q$.    Therefore, the Brownian motions obtained from the   expansion are no longer independent. This is both the  difficulty and key point in this paper. 

Regarding   the approximation  scheme with respect to   spatial variables, {  motivated by  \cite{Zhangxuping2013}}   we truncated the infinite expansion in terms of the two bases, corresponding to   the eigenvectors of the Dirichlet Laplacian  and the noise, 
respectively, in which the  dimension of the Dirichlet Laplacian  and Gaussian noise does  not need to be equal. 
A  scheme is given as follows (the notations are given in Section \ref{s.spatial})
\begin{equation*}
	\d  U_{N,M}(t)= -(\La^N + B^N)  U_{N,M}(t) \d t +\beta_1 \sum_{j=1 }^M     A^N_j U_{N,M}(t)\d B_j(t).
\end{equation*}
Under the stable condition in Theorem \ref{exactstability} and some conditions  on the noise covariance, the spectral  approximation  scheme  is long time stable
(e.g. Theorem \ref{Theoremfinite}).  In addition, the convergence of  spectral  approximation  scheme  is also obtained
(Theorem \ref{convergencetheorem}), i.e.,
\[
\sup_{0\le t\le T}   	\||u(t)-U_{N,M}(t)\||_2^2\le C(\lambda_N^{-1+\gamma}+\rho(M))\,.  
\]

To further investigate the stability of   fully temporal-spectral approximation  scheme,   we mainly consider the implicit Euler method, described by
\begin{equation*}
	U_{n+1}= U_n - \tau \La^N U_{n+1} - \tau B^N U_{n+1}+\beta_1 \sum_{j=1}^M  A^N_j U_n\Delta B_j^n,
\end{equation*}
for which  the mean-square stability is again obtained
(Theorem \ref{Heattheorem}). In addition, some other numerical methods and corresponding stability results are also discussed.

There are some other  papers on the  study of stability for   general stochastic partial differential
equations.  First, there are a lot of results in the abstract framework.
For example, \cite{Baojianhai2011} discuss the   asymptotic stability of the exact solutions. The polynomial stability of stochastic heat equations is obtained in \cite{Lv2023}.  Lang et. al. investigate in \cite{Lang2017} the numerical mean-square stability of the stochastic partial differential
equations driven by U-valued Q-L\'evy process. 
Since they are in the abstract  framework, it is hard to apply verify
their condition in our concrete equations. As we indicate earlier,  the closely relevant to our
work is \cite{Baojianhai2011} and we made some comparison in Remark 2.2. 
Let us also mention that there are also works on    the numerical almost sure exponential stability condition for  a stochastic heat equation driven by Brownian   motion 
which depends on time only and does not depend on space variables 
(see e.e. Yang et  al. \cite{Yang2022}).

The paper is organized as follows. In Section 2,  in view of the eigenfunction expansion method, we discuss the stability condition of the exact solution. In Section 3,  the stability and convergence results are  given for the spectral  approximation   scheme. In Section 4, we apply the implicit Euler method  for the temporal discretization and obtain its stability. In Section 5, we present two numerical experiments to illustrate the theoretical findings.

\section{Long time stability of the solution}\label{s.2} 

Let $L^2(\mathcal{O})$ be the space  of all square integrable functions on  bounded   domain   $\mathcal{O}$ 
of $\RR^d$ with smooth boundary $\partial D$.  The usual inner product  and norm are denoted by  $\langle \cdot,\cdot \rangle$  and $\|\cdot\|_2$. Let $H_0^1(\mathcal{O})$ and $H^2(\mathcal{O})$ be the usual Sobolev spaces with and  without compact supports. 
Let $0<\la_1\le \la_2\le \cdots$ 
be the eigenvalues of $-\De$ with  Dirichlet 
boundary condition and let $e_k\in L^2(\cO)\,, k=1, 2, \cdots$ be the 
corresponding eigenfunctions,  
namely,
\[
\De e_k(x)=-\la_k e_k(x)\,,\quad x\in \cO\,,\ \ k=1, 2, \cdots.
\]
Since $\cO$ is a bounded open domain  with smooth boundary we also know that $\{e_1, e_2, \cdots\}$ constitutes an  
orthonormal basis of $L^2(\cO)$  and   (see \cite[Theorem 6.3.1]{Davies1995}
or \cite{Davies1990})
\begin{equation}
	C_1k^{2/d}\le \la_k\le C_2k^{2/d}\, 
\end{equation}
for two positive constants $C_1, C_2$.   With  this basis we can 
correspond a function $f(x), x\in \cO$,  in $L^2(\cO)$  to  an
infinite dimensional vector $f\leftrightarrow 
(f_1, f_2, \cdots)$,   where $f_k=\langle f, e_k):=\int_{\cO} f(x) e_k(x) dx$.
Namely, $f=\sum_{k=1}^\infty f_k e_k$, with  the convergence being  in $L^2(\cO)$.   
In other word, we can write   
\[
u(t,x)=\sum_{k=1}^\infty u_k(t) e_k(x)\,,
\]
where 
\[
u_k(t)=\langle u(t, \cdot), e_k\rangle =\int_\cO u(t,x)e_k(x)d x.
\]
Thus, we can regard $u(t,x), x\in \cO$ as an  infinite 
dimensional
vector $u(t,x)\leftrightarrow (u_1(t), u_2(t), \cdots) 
$.  We shall   write   \eqref{Heatequation} as an infinite dimensional
stochastic (ordinary) differential equation along  the above spirit. 
For this reason we also expand 
\[
\dot W(t,x)=\sum_{k=1}^\infty \langle \dot W(t, \cdot)\,, e_k) e_k(x) 
\]
along the basis of   eigenfunctions of $\De$.   
To understand better this expansion, we 
denote 
\[
\dot B_k(t)=\langle \dot W(t, \cdot)\,, e_k)=\int_{\cO} \dot W(t,x)e_k(x) dx\,,\quad 
k=1, 2, \cdots
\]
which is family of a mean zero Gaussian processes of time $t$.
The covariance of this family of processes are given by
\begin{equation}
	\begin{split}
		\EE \left[ \dot B_i(t)\dot B_j(s)\right]
		=& \EE \int_{\cO^2} \dot W(t,x)e_j(x)\dot W(t,y)e_j(y) dxdy\\
		=&   \int_{\cO^2} \EE \left[ \dot W(t,x) \dot W(t,y)\right] e_i(x)e_j(y) dxdy\\
		=&   \delta(t-s) \int_{\cO^2} q(x,y) e_i(x)e_j(y) dxdy
		=\delta(t-s)  \al_{ij}\,, 
	\end{split}
\end{equation}
where and  throughout this paper 
\begin{equation}\label{e.2.2} 
	\al_{ij}:= \int_{\cO^2} q(x,y) e_i(x)e_j(y) dxdy\,. 
\end{equation}
Thus, $B_i(t)=\int_0^t \dot B_i(s)ds, i=1, 2, \cdots$ is 
a family of   Brownian motions with covariance given by $\al_{ij}$.

To write Equation  \eqref{Heatequation} in the infinite dimensional vector form, we   need to write $u\dot W$ in the infinite dimensional vector form. 
First,  we recall  from \cite[Chapter 6, Section 5, Theorem 1]{Evans2010} that the eigenfunctions    are in $C^\infty(\bar \cO) $ since $\partial \cO$ is smooth.  
Since $\bar \cO  $ is compact  all $e_i$'s  are then bounded functions   and hence  $e_ie_j\in L^2(\cO)$. Thus, we have 
\begin{equation}
	e_i(x)e_j(x)=\sum_{k=1}^\infty a_{jki} e_k(x)\,, 
\end{equation}  
where
\begin{equation}\label{e.2.5} 
	a_{jki}: =\langle e_ie_je_k\rangle=\langle e_ie_j, e_k\rangle
	=\int_{\cO} e_i(x)e_j(x)e_k(x) dx\,. 
\end{equation}
We also point out that $a_{jki} $ is permutation invariant with respect to $(k,i,j)$. 

With this notation, we have 
\begin{equation}
	\begin{split}
		u(t,x)\dot W(t,x)
		=&\sum_{i,j=1}^\infty u_i(t) e_i(x) \dot B_j(t) e_j(x)\\
		=&\sum_{i,j,k=1}^\infty a_{jki} u_i(t) \dot B_j(t) e_k(x)\,. 
	\end{split}
\end{equation}
Therefore, the equation \eqref{Heatequation}  becomes 
\begin{equation*}
	\begin{split}
		\sum_{k=1}^\infty u_k'(t) e_k(x) =&
		-\sum_{k=1}^\infty \la_k u_k (t) e_k(x)
		-\beta_1 \sum_{k=1}^\infty   u_k (t) e_k(x)\\ 
		&+\sum_{i,j,k=1}^\infty a_{jki} u_i(t) \dot B_j(t) e_k(x)\,. 
	\end{split}
\end{equation*}
Or  can write    \eqref{Heatequation}  as   following (infinite 
dimensional)  stochastic ordinary equation: 
\begin{equation*}
	\begin{split}
		u_k'(t)   =&
		- \la_k u_k (t)  
		-\beta_0    u_k (t)  
		+\be_1\sum_{i,j =1}^\infty a_{jki} u_i(t) \dot B_j(t)  \,, \quad k=1, 2, \cdots  
	\end{split}
\end{equation*}
If we denote 
\[
u(t)=(u_1(t), u_2(t), \cdots)^T\,, \quad \La =\diag(\la_1, \la_2, \cdots)\,,
\]
and
\begin{equation}
	A_j=\left( a_{jki}\right)_{1\le k,i<\infty}\,, \quad j=1,2, \cdots
	\label{e.2.5a} 
\end{equation} 
Then we can write \eqref{Heatequation}  as a more compact form: 
\begin{equation}
	du(t)=  -\La u(t) dt -\be_0 u(t) dt  +\be_1\sum_{j=1}^\infty A_j u(t) dB_j(t)\,. 
	\label{e.2.6} 
\end{equation} 
Now we specify the Hilbert space that the finite dimensional
vector for the solution belongs.  Let $H=(H_{ij})_{1\le i,j<\infty} $ be a symmetric strictly positive define infinite matrix and without ambiguity we use the same notation 
$H$ to represent the Hilbert space 
\[
H=\left\{ a=(a_1, a_2, \cdots)^T\,;\ \|a||_H^2:=a^THa=\sum_{i,j=1}^\infty a_iH_{ij}a_j<\infty\right\}\,.
\] 
We can study the stability of $u$ with respect to the (semi-)norm defined
by this matrix, namely, the stability of $\|u(t)\|_H^2=u^T(t)Hu(t)$.
From \eqref{e.2.6} and by It\^o's formula, we have
\[
\begin{split}
	d u^T(t)Hu(t)
	=&u^T(t)Hdu(t)+(du^T(t)) Hu(t)+(du^T(t)) H du(t)\\
	=&u^T(t)H\left[ -\La u(t) dt -\be_0 u(t) dt  +\be_1\sum_{j=1}^\infty A_j u(t) dB_j(t)\right]\\
	&\qquad  +\left[ -  u^T(t)\La dt -\be_0 u^T(t) dt  +\be_1\sum_{j=1}^\infty  u^T(t)A_j ^T dB_j(t)\right]Hu(t)\\
	&\qquad + \be_1^2\sum_{i,j=1}^\infty u^T(t) A_i^THA_j u(t) \al_{ij} dt\\
	=& \bigg[ -u^T(t) \left[ H \La+\La H\right] u(t) -2\be_0 u^T(t) Hu(t)\\
	&\qquad + \be_1^2 \sum_{i,j=1}^\infty \al_{ij} u^T(t) A_i^T H A_j u(t)\bigg] dt \\
	&\qquad + \be_1 \sum_{j=1}^\infty \left[ u^T(t) HA_j u(t)+ u^T(t) A_j^T Hu(t)\right] dB_j(t)  \,. 
\end{split}
\]
Taking the expectation yields
\begin{equation}
	\begin{split}
		\frac{d}{dt}  \EE \left[u^T(t)Hu(t)\right] 
		=& -\EE \left[ u^T(t) \left[ H \La+\La H\right] u(t)\right]  -2\be_0 \EE \left[ u^T(t) Hu(t)\right]\\
		&\qquad + \be_1^2 \sum_{i,j=1}^\infty \al_{ij} \EE \left[  u^T(t) A_i^T H A_j u(t)\right]   \,.
	\end{split}
	\label{e.2.7} 
\end{equation} 
To concisely present our idea let us consider the case that 
$H=I$ is the identity matrix, namely, $H_{ij}=\de_{ij}$ for all $i,j=1, 2,  \cdots$,
where $\de_{ij}=1$ when $i=j$ and $\de_{ij}=0$ when $i\not=j$ is the Kronecker symbol.   Then 
\[
\|u(t)\|_H^2=  u^T(t)Hu(t)  =  \sum_{j=1}^\infty  u_j^2 (t)=\int_{\cO}  |u(t,x)|^2 dx=\|u(t)\|_2^2\ 
\]
is the usual $L^2$ norm.  
In this case, we have
\begin{equation}
	\begin{split}
		u^T(t) \left[  H \La+\La H\right] u(t) 
		&=2\sum_{j=1}^\infty\la_j  u_j^2(t) \ge 2\la_1
		\sum_{j=1}^\infty   u_j^2(t)=2\la_1 \|u(t)\|_2^2\,.  
	\end{split}\label{e.2.8} 
\end{equation}
Using the fact that 
\[
\sum_{i=1}^\infty e_i(x)e_i(y)=\delta(x-y)\,,
\]
where $\delta$ is the Dirac delta function 
whose justification can always be done via approximation,   and by the definition of $\al_{ij}$ (i.e.  \eqref{e.2.2}) and $A_j$ (i.e. equations \eqref{e.2.5} and \eqref{e.2.5a}), we  have 
\begin{equation}
	\begin{split}
		\bigg( \sum_{i,j=1}^\infty &\al _{ij}    A_i^T    A_j  \bigg)_{k,m}
		=  \sum_{i,j,\ell=1}^\infty \al_{ij}     a_{i\ell k}  a_{j\ell m} \\
		=& \sum_{i,j,\ell=1}^\infty   \int_{\cO^4} q(x,y) e_i(x) e_j(y) 
		e_i(\xi) e_\ell (\xi) e_k(\xi) e_j(\eta) e_ \ell(\eta) e_m(\eta) dxdyd\xi d\eta\\ 
		=&    \int_{\cO^4} q(x,y) \de(x-\xi) \de(y-\eta) 
		\delta(\xi-\eta)
		e_k(\xi)     e_m(\eta) dxdyd\xi d\eta\\
		=& \int_{\cO} q(\xi,\xi)   e_k(\xi)     e_m(\xi)  d\xi\,. 
	\end{split}\label{e.2.9a} 
\end{equation}
Therefore, we have 
\begin{equation}\label{e.2.9} 
	\begin{split}
		\sum_{i,j=1}^\infty \al_{ij}  \left[  u^T(t) A_i^T   A_j u(t)\right]  
		=& \sum_{k, m=1}^\infty  \int_{\cO} q(\xi,\xi)   u_k(t)
		e_k(\xi)     u_m(t) e_m(\xi)  d\xi\\
		=& \   \int_{\cO} q(\xi,\xi)  \left( \sum_{k =1}^\infty      u_k(t)
		e_k(\xi)    \right)^2   d\xi =   \int_{\cO} q(\xi,\xi)   u ^2(t, \xi)    d\xi\\
		\le& \kappa   \int_{\cO}    u ^2(t, \xi)    d\xi  \,, 
	\end{split}
\end{equation}
where
\begin{equation}
	\kappa =\sup_{\xi \in \cO} |q(\xi, \xi)|\,. \label{e.2.10} 
\end{equation}
Denote 
\begin{equation}
	\||u(t)\||_2^2=\EE \|u(t)\|_2^2\,.\label{e.2.11} 
\end{equation}
Plugging  \eqref{e.2.8}-\eqref{e.2.9} into \eqref{e.2.7}  yields 
\begin{equation}
	\frac{d}{dt} \||u(t)\||_2^2 \le -2(\la_1+\beta_0) \||u(t)\||_2^2 
	+ \be_1^2 \kappa \||u(t)\||_2^2 \,.\label{e.2.13a}
\end{equation}
An application of the Gronwall inequality yields
\begin{equation}
	\||u(t)\||_2^2 \le \exp\left(  -[2(\la_1+\beta_0)- \be_1^2  \kappa]  \right)  \||u(0)\||_2^2 \,.\label{e.2.13aa}
\end{equation} 
We can summarize the above result as
\begin{theorem}\label{exactstability} 
	Let $\kappa$ be given by  \eqref{e.2.10}.  If  
	\begin{equation}
		2(\la_1+\be_0) -\be_1^2\kappa>0\,,\label{e.2.12} 
	\end{equation}
	then the system \eqref{Heatequation} is stable in the mean square sense, namely, 
	$\lim_{t\to \infty} \||u(t)\||_2=0$.   
\end{theorem}  
\begin{proof}
	This is a consequence of \eqref{e.2.13aa}. 
\end{proof} 

\begin{remark}
	\begin{enumerate}
		\item[(1)]  The condition of this theorem requires that  the covariance 
		function $q(x,y)$  satisfies 
		$\kappa =\sup_{\xi \in \cO} |q(\xi, \xi)|<\infty$. This condition is quite  restrictive
		and  is not satisfied by the  noises   
		studied recently in the literature concerning intermittency.   For instance,   when the noise is
		fractional Gaussian, namely, 
		$q(x,y)=\prod_{i=1}^d |x_i-y_i|^{2H_i-2}$, the quantity $\kappa=\infty$,   where $H_i\in (0, 1), i=1, \cdots, d$ are the Hurst parameters.  On the other hand,  
		for the fractional Brownian field $q(x,y)=\prod_{i=1}^d \left[ |x_i|^{2H_i}+
		|y_i|^{2H_i}-|x_i-y_i|^{2H_i }\right] $,  the quantity 
		$\kappa =\sup_{\xi \in \cO} |q(\xi, \xi)|$ is finite.  This phenomenon  
		is natural since we require the solution to be long time stable. 
		%
		%
		\item[(2)] A different  condition may also be obtained by applying the 
		result of \cite{Baojianhai2011} and then  the   condition for stability
		in \cite{Baojianhai2011}  implies  
		\begin{equation}\label{Heatcondition}
			2 \lambda_1-2 \beta_0-\beta_1^2 C^2 \sum_{n=1}^{\infty} q_n>0 ~~~\text{(in our notations)}
		\end{equation}
		where  $q_n$ are the eigenvalues of 
		covariance operator:  
		$q(x,y)=\sum_{n=1}^\infty q_n \xi_n(x) \xi_n(y)$ for some orthonormal basis 
		$L^2(\cO)$,  and  $C:=\sup _{n \in \mathbb{N}, \theta \in \mathcal{O}}\left|\xi_n(\theta)\right|<\infty$.   It is obvious  that the condition \eqref{Heatcondition} is stronger than the  condition 
		\eqref{e.2.12} and latter is also much easier to verify.  For example,   when 
		$q(x,y)=\prod_{i=1}^d \left[ |x_i|^{2H_i}+
		|y_i|^{2H_i}-|x_i-y_i|^{2H_i }\right] $ it seems very hard to find  its associated  eigenvalues and eigenfunctions.   But in this case   
		\[ \kappa\le   2^d \left( 
		\sup_{1\le i\le d} \sup_{x\in \cO} |x_i|\right)^{2( H_1+\cdots+H_d)} \,. 
		\]  
	\end{enumerate}
\end{remark}

\section{Spectral  Approximation} \label{s.spatial} 
We  continue to study the equation \eqref{e.2.6} in its component form 
obtained in previous section: 
\begin{equation}
	\d u_i(t ) =-\left[ \la_i+\beta_0\right]  u_i(t) \d t   +\beta_1 \sum_{j,k=1}^\infty  a_{jki} u_k(t)  \d B_j(t),   
	\label{e.3.1} 
\end{equation}
$i=1, 2, \cdots$.   The associated   eigenfunction expansion for the solution is 
\begin{equation*}
	u(t,x)= \sum_{l=1}^{\infty} u_l(t)  e_l(x),
\end{equation*}
  where $u_l(t)=\langle e_l(x),  u(t,x) \rangle$.

We denote   by $S_N$ the space spanned by  the first $N$ orthonormal eigenvectors $\{e_1, \cdots, e_N\}$
of  the Laplace operator.
Let $P_N: L^2(\mathcal{O})\rightarrow S_N$ be the orthogonal projection onto  $S_N$.
Thus we have 
\[\langle u-P_Nu, \phi \rangle=0, ~~~~\text{for}~~~~\phi \in S_N\,. \]
It is  known  that $P_Nu$ approximates $u$ in the $L^p$-norms  
as $N\to \infty$ (e.g.   \cite{Canuto1988}).
With the definition of space $S_N$, $  P_N u$ is the truncated series of $u$, i.e.,
\begin{equation*}
	P_Nu(t,x)= P_N(\sum_{l=1}^{\infty} u_l(t)  e_l(x))=\sum_{l=1}^{N} u_l(t)  e_l(x).
\end{equation*}
We approximate the solution $u(t,x)$ by   truncating   equation   \eqref{e.3.1} as follows. 
\begin{equation}
	\d u_{N,i} (t ) =-\left[ \la_i+\beta_0\right]  u_{N,i}(t) \d t   +\beta_1 \sum_{j=1 }^\infty   \sum_{k=1}^N a_{jki} u_{N,k} (t)  \d B_j(t)\,,     
	\label{e.3.2} 
\end{equation}
$i=1, 2, \cdots, N$.  Denoting $U_{N}(t)=(u_{N,1}(t),\cdots,u_{N,k} (t),\cdots,u_{N,N} (t))^T$,  
\eqref{e.3.2} can   then be rewritten in a matrix notation    as   
\begin{equation}\label{HeatSemimatrix}
	\d  U_{N}(t)= -(\La^N + B^N)  U_{N}(t) \d t +\beta_1 \sum_{j=1 }^\infty     A^N_j U_{N}(t)\d B_j(t),
\end{equation}
where $\La^N=\diag (\lambda_1,\cdots,\lambda_k,\cdots,\lambda_N)$, $B^N=\diag (\beta_0,\cdots,\beta_0,\cdots,\beta_0)$ and 
$A^N_j=(\langle e_je_ie_k\rangle )_{1\le i,k\le N}$.  
Let us point out that $U_N(t)$ is in general not the projection of $u(t)$ to $S_N$.  
However,  we have the following long time  stability result for the approximated solution $U_N(t )$.  
\begin{theorem} \label{theoremsemi1}
	Let \begin{equation} 
		\begin{split}	
			\|q_N\|_2: =  \sup_{\|f\|_2\le 1} \int_{\cO^2}q_N(\xi,\eta) f(\xi)f(\eta)\d\xi \d\eta\,,\quad 	\tilde \kappa_1:=\sup_{N\ge 1} \|q_N\|_2<\infty \,. 
		\end{split}
	\end{equation}  
	where $	 q_N(\xi, \eta)=   \sum_{\ell=1}^N q(\xi, \eta)e_\ell (\xi) e_\ell(\eta)  $.
	If  
	\begin{equation}\label{stabilityconditionsemi}
		2(\la_1+\be_0) -\be_1^2  \tilde \kappa_1>0\, ,
	\end{equation}
	then the system \eqref{HeatSemimatrix} is stable in the mean square sense, namely, 
	$$\lim_{t\to \infty} \sup_{N\ge 1} \||U_{N}(t)\||_2=0.$$   
\end{theorem} 
\begin{proof}
	In view of  the It\^o formula, we know
	\[
	\begin{split}
		\d U_N^T(t)U_N(t)
		=&U_N^T(t)\d U_N(t)+(\d U_N^T(t)) U_N(t)+(\d U_N^T(t))  \d U_N(t)\\
		=& \bigg[ -2 U_N^T(t) (\La^N+B^N) U_N(t)  + \be_1^2 \sum_{i,j=1}^\infty \al_{ij} U_N^T(t) (A^N_i)^T A^N_j U_N(t)\bigg] \d t \\
		&\qquad + \be_1 \sum_{j=1}^\infty \left[ U_N^T(t) A^N_j U_N(t)+ U_N^T(t) (A^N_j)^T U_N(t)\right] \d B_j(t)  \,. 
	\end{split}
	\] 
	Taking the expectation on both sides of the above equation gives
	\begin{equation}
		\begin{split}
			\frac{\d}{\d t}  \EE \left[U_N^T(t)U_N(t)\right] 
			=& -2\EE \left[ U_N^T(t) (\La^N+B^N) U_N(t)\right] + \be_1^2 \sum_{i,j=1}^\infty \al_{ij} \EE \left[  U_N^T(t) (A^N_i)^T A^N_j U_N(t)\right].
		\end{split}\label{e.3.5} 
	\end{equation} 
	We now estimate each of the terms on the right side of the above equation. For the first one, it is obvious to see from (\ref{stabilityconditionsemi}) that the matrix $\La^N+B^N$ is symmetric positive definite   with the smallest eigenvalue $\la_1+\be_0$. Consequently,
	\begin{equation}
		\EE \left[ U_N^T(t) (\La^N+B^N) U_N(t)\right] \geq (\lambda_1+\be_0)\EE \left[ U_N^T(t)  U_N(t)\right].\label{e.3.6} 
	\end{equation}  
	Before giving the estimation of the second term on the right hand,  we notice 
	\begin{equation}
		\begin{split}
			\bigg( \sum_{i,j=1}^\infty &\al _{ij}   ( A_i^N)^T    A_j^N  \bigg)_{k,m}
			=  \sum_{i,j=1}^\infty \sum_{\ell=1}^N \al_{ij}     a_{i\ell k}  a_{j\ell m} \\
			=&    \sum_{\ell=1}^N  \sum_{i,j=1}^\infty \int_{\cO^4} q(x,y) e_i(x) e_j(y) e_i(\xi) 
			e_\ell (\xi)  e_j(\eta) e_ \ell(\eta)
			e_k(\xi)     e_m(\eta) dxdyd\xi d\eta\\	
			=&    \int_{\cO^4} q(x,y) \de(x-\xi) \de(y-\eta) 
			e_\ell (\xi)  e_ \ell(\eta)
			e_k(\xi)     e_m(\eta) dxdyd\xi d\eta\\
			=&{   \sum_{\ell=1}^N \int_{\cO^2} q(\xi,\eta)  	e_\ell (\xi)  e_ \ell(\eta) e_k(\xi)     e_m(\eta)  d\xi d \eta.}
		\end{split}\label{e.alA}
	\end{equation} 
	To apply the properties of orthonormal basis in the infinite dimensional 
	space, we identify   a   vector $  u=(  u_1, \cdots,  u_N)^T$ in $S_N$ as an infinite dimensional  vector $\hat{u}  =(\hat{u}_1  ,\cdots,\hat{u}_i ,\cdots)^T$,
	in which $\hat{u}_i =\hat u_i $ for $i\leq N$ and $\hat{u}_i =0$ for $i>N$. Clearly,  
	the $L^2$ norms of these two identified vectors are the same,  
	i.e., $\|u \|_2=\|\hat u\|_2$. 
	Thus by \eqref{e.alA}   and the definition of $\kappa$  by (\ref{e.2.10}) we have 
	\begin{equation}
		\begin{split}
			\sum_{i,j=1}^\infty  \al_{ij} \left[  U_N^T(t) ( A_i^N)^T    A_j^N  U_N(t)\right]  
			=& \sum_{k, m,\ell=1}^N  \int_{\cO^2} q(\xi,\eta)   u_k(t)
			e_k(\xi)     u_m(t) e_m(\eta) e_\ell (\xi)  e_ \ell(\eta) \d \xi \d \eta\\
			=& \sum_{k, m}^\infty  \sum_{\ell=1}^N \int_{\cO^2} q(\xi,\eta)   \hat{u}_k(t)
			e_k(\xi)     \hat{u}_m(t) e_m(\eta)  e_\ell (\xi)  e_ \ell(\eta) \d \xi \d \eta \\
			=&  \int_{\cO^2} q_N(\xi,\eta)   \left( \sum_{k=1}^\infty \hat{u}_k(t)
			e_k(\xi) \right)  \left( \sum_{m=1}^\infty \hat{u}_m(t)
			e_m(\xi) \right)      \d \xi\d\eta  \\ 
			=&  \int_{\cO} q_N(\xi,\eta)   \hat{U}_N(t, \xi) \hat{U}_N(t, \eta)    \d \xi
			\d \eta\\
			\le&\tilde \kappa_1  \int_{\cO}    (\hat{U}_N(t, \xi) ) ^2   \d \xi  \,. 
			\label{e.3.7} 
		\end{split}
	\end{equation}
	Substituting \eqref{e.3.6} and \eqref{e.3.7} into \eqref{e.3.5} yields 
	\begin{equation*}
		\frac{d}{dt} \||U_N(t)\||_2^2 \le -2(\la_1+\beta_0) \||U_N(t)\||_2^2 
		+ \be_1^2 \tilde \kappa_1 \||U_N(t)\||_2^2 \,.
	\end{equation*}
	The proof of the theorem is hence completed by an application of the Gronwall lemma.
\end{proof}

\begin{remark}\label{r.3.2} 
	  When $N\to \infty$,  we have  
	\[ 	 q_N (\xi, \eta)=   \sum_{\ell=1}^N  q(\xi, \eta)e_\ell (\xi) e_\ell(\eta) \to q(\xi, \eta)\delta(\eta-\xi) \,.
	\]
Thus,   \begin{equation} 
	\begin{split}	
		\|q_N\|_2: =&  \sup_{\|f\|_2\le 1} \int_{\cO^2}q_N(\xi,\eta) f(\xi)f(\eta)\d\xi \d\eta \\
		\to& \sup_{\|f\|_2\le 1} \int_{\cO^2}q (\xi,\eta) f(\xi)f(\eta)
		\delta(\eta-\xi) \d\xi \d\eta \\
		=&\sup_{\|f\|_2\le 1} \int_{\cO} q (\xi,\xi) f^2(\xi)d\xi \\
			=&\sup_{g\ge 0, \|g\|_1\le 1} \int_{\cO} q (\xi,\xi) g(\xi)d\xi=\sup_{\xi \in \cO}  q(\xi, \xi) =\kappa   
	\end{split}
\end{equation}  
since $q(\xi, \xi)\ge 0$.  This implies that 
$\kappa\le \tilde \kappa_1$.  Thus the condition
\eqref{stabilityconditionsemi} in Theorem \ref{theoremsemi1}  is stronger than
the condition
\eqref{e.2.12} in Theorem \ref{exactstability}.   
	\end{remark}

In \eqref{e.3.2} or  \eqref{HeatSemimatrix},  we expand  the  Gaussian noise into an infinite series so that the equation contains infinitely  many 
(correlated) Brownian motions.   Now  we want to  consider the truncation  of  the noise  in (\ref{HeatSemimatrix}).   We also allow the truncation of   solution  and that of the noise to contain different number of terms.  More precisely, we shall consider the following approximation:  
\begin{equation}\label{HeatSemimatrix1}
	\d  U_{N,M}(t)= -(\La^N + B^N)  U_{N,M}(t) \d t +\beta_1 \sum_{j=1 }^M     A^N_j U_{N,M}(t)\d B_j(t).
\end{equation} 
As for the truncation  of   solution the  truncation of Gaussian noise
is also identified with   an infinite dimensional vector: For any positive integer $M$, 
\[
\dot{W}_M(t,x)=\sum_{i=1}^M \dot{B}_{ i}(t)e_i(x)+\sum_{i=M+1}^\infty 0e_i(x)\triangleq \sum_{i=1}^\infty \dot{\cB}_{M,i}(t)e_i(x)\,, 
\]
where $\dot{\cB}_{M,i} (t)=\dot{B}_i(t) $ for $i\le M$ and 
$\dot{\cB}_i(t)=0 $ for $i> M$.  
The covariance of   truncated Gaussian noise is given by
\begin{equation*}
\begin{split}
	\EE \left[ \dot W_M(t,x)\dot W_M(s,y)\right]
	=& \sum_{i,j=1}^M \EE  (\dot B_i(t)\dot B_j(s))e_i(x)e_j(y) \\
	=& \hat{q}_M(x,y) \delta(t-s),
\end{split}
\end{equation*}
where $\hat{q}_M(x,y)=\sum_{i,j=1}^M \al_{ij}e_i(x)e_j(y)$.
\begin{equation}\label{truncationnoise}
\begin{split}
	\EE \left[ \dot \cB_{M,i}(t)\dot \cB_{M,j}(s)\right]
	=& \EE \int_{\cO^2} \dot W_M(t,x)e_j(x)\dot W_M(t,y)e_j(y) \d x\d y\\
	=&   \int_{\cO^2} \EE \left[ \dot W_M(t,x) \dot W_M(t,y)\right] e_i(x)e_j(y) \d x \d y\\
	= & \delta(t-s) \hat{\al}_{M,ij},
\end{split}
\end{equation}
where $\hat{\al}_{M,ij}= \int_{\cO^2}\hat{q}_M(x,y) e_i(x)e_j(y) \d x\d y$.

For this approximation we have the following   the stability result. 
\begin{theorem}\label{Theoremfinite}
Let  \begin{equation}
	\begin{split}
		\hat{q}_{M,N}(\xi,\eta)=
		& \sum_{\ell=1}^N   \tilde q_M(\xi, \eta)  
		e_{\ell } (\xi)  e_\ell (\eta)  \\
		=& \sum_{\ell=1}^N  \sum_{i,j=1}^M \int_{\cO^2} q(x,y)e_i(x)e_j(y)dxdy 
		e_{i} (\xi) e_{i} (\eta)  e_{\ell} (\xi) e_{\ell}(\eta)    
	\end{split}
\end{equation}  
and let
\begin{equation}
	\tilde  \kappa_2= \|\hat{q}_{M,N}\|_2 \,.
\end{equation}	 
If  
\begin{equation}
	2(\la_1+\be_0) -\be_1^2\tilde \kappa_2>0\,,
	\label{e.3.15} 
\end{equation}
then the system \eqref{HeatSemimatrix1} is stable in the mean square sense, namely, 
$$\lim_{t\to \infty} \sup_{N,M \ge 1} \||U_{N,M}(t)\||_2^2=0. $$  
\end{theorem} 
\begin{proof} We can also extend $ {U}_{N,M}(t)$ to an infinite dimensional vector $\hat{U}_{N,M}(t)$ as before. 
Using the notation $\cB_{M,j}(t)$ we can rewrite  \eqref{HeatSemimatrix1}   as  
\begin{equation*}
	\d  U_{N,M}(t)= -(\La^N + B^N)  U_{N,M}(t) \d t +\beta_1 \sum_{j=1 }^\infty     A^N_j U_{N,M}(t)\d \cB_{M,j}(t),
\end{equation*}
in which  the relation  $\|\hat{U}_{N,M}(t)\|_2=\|{U}_{N,M}(t)\|_2$ is still holds.

Similarly to the proof of  Theorem \ref{theoremsemi1}   
we can obtain
\begin{equation*}
	\begin{split}
		\sum_{i,j=1}^\infty \hat{\al}_{M,ij}  \left[  U_{N,M}^T(t) (A_i^N)^T   A_j^N U_{N,M}(t)\right]  
		=& \sum_{k, m,\ell=1}^N  \int_{\cO^2} \hat{q}_M(\xi,\eta)   u_{N,M,k}(t)
		e_k(\xi)     u_{N,M,m}(t) e_m(\eta) e_\ell (\xi)  e_ \ell(\eta) \d\xi \d\eta\\
		=&     \int_{\cO^2} \hat{q}_{M,N}(\xi,\eta)   \left(\sum_{k  =1}^{ {\infty}} \hat{u}_{N,M,k}(t)
		e_k(\xi) \right)   \left(\sum_{  m =1}^{ {\infty}}\hat{u}_{N,M,m}(t) e_m(\eta)\right)    \d\xi \d\eta  \\
		=&  \int_{\cO} \hat{q}_{M,N} (\xi,\eta)  \hat{U}_{N,M}(t, \xi)  \hat{U}_{N,M}(t, \eta)  \d\xi \delta\eta \\
		\le& \tilde \kappa_2   \int_{\cO}    (\hat{U}_{N,M}(t, \xi) ) ^2   \d\xi  \,, 
	\end{split}
\end{equation*}
Therefore, 
\begin{equation*}
	\frac{d}{dt} \||U_{N,M}(t)\||_2^2 \le -2(\la_1+\beta_0) \||U_{N,M}(t)\||_2^2 
	+ \be_1^2 \tilde \kappa_2 \||U_{N,M}(t)\||_2^2 \,.
\end{equation*}
The proof is hence completed.
\end{proof}

\begin{remark} Notice 
	\begin{equation}
		\begin{split}
			\hat{q}_{M,N}(\xi,\eta) 
			=& \sum_{\ell=1}^N  \sum_{i,j=1}^M \int_{\cO^2} q(x,y)e_i(x)e_j(y)dxdy 
			e_{i} (\xi) e_{i} (\eta)  e_{\ell} (\xi) e_{\ell}(\eta)\\
			  \to& \sum_{\ell=1}^N    q(\xi,\eta )  e_{\ell} (\xi) e_{\ell}(\eta) 
			  \quad \hbox{when   $M\to \infty$}\,.   
		\end{split}
	\end{equation}  
As in Remark \ref{r.3.2}, we see that $\tilde \kappa_2\le \tilde \kappa_1$. Thus,   condition \eqref{e.3.15} in the above theorem is stronger than 
\eqref{stabilityconditionsemi} in Theorem \ref{theoremsemi1}.  
	\end{remark}
	
In this following, we are focus on the convergence of the numerical scheme. Before giving the theorem, we review the the generalized Gronwall lemma in \cite{limin2021}. 
\begin{lemma}\label{lemma}
Assume that $\alpha>0, \beta>0, \alpha+\beta<1$ and $a \geq 0, b \geq 0$. Suppose that $u$ is a nonnegative function satisfying that $t^{-\beta} u(t)$ is locally integrable on $\mathbb{R}_{+}$.
If u satisfies
$$
u(t) \leq a+b \int_0^t(t-s)^{-\alpha} s^{-\beta} u(s) \d s\,, \ \ \quad \forall t \in \mathbb{R}_{+},
$$
then
$$
u(t) \leq a E_{1-\alpha, 1-\beta}\left((b \Gamma(1-\alpha))^{1 /(1-\alpha-\beta)} t\right)\,, \ \  \quad \forall t \in \mathbb{R}_{+},
$$
where $E_{1-\alpha, 1-\beta}(s)=O\left(s^{\frac{1}{2}\left(\frac{1-\alpha-\beta}{1-\alpha}-1+\beta\right)} \exp \left(\frac{1-\alpha}{1-\alpha-\beta} s^{\frac{1-\alpha-\beta}{1-\alpha}}\right)\right)$ is the generalized Mittag-Leffler function. 
\end{lemma}

We also need the following inequality.
\begin{lemma} Let $c_0>0$ and $\beta>0$.  Thus there is a constant 
	$C$ depending on $c_0$  and $\beta$  but independent of $\tau>0$   such  that
\begin{equation}\label{e.3.17} 
	\sum_{k=1}^\infty e^{-c_0k^\be  \tau }\le C/\sqrt{\tau}\,,
	\quad \forall \tau>0\,. 
\end{equation}	
	\end{lemma} 
\begin{proof} It is obvious 
\[
\begin{split}
\sum_{k=1}^\infty e^{-c_0k^\be  \tau } \le & 
 \sum_{k=1}^\infty \int_{k-1}^{k } e^{-c_0x^\be  \tau } dx =\int_0^\infty e^{-c_0x^\be  \tau } dx   \\
	=& \tau^{-1/\beta} \int_0^\infty e^{-c_0x^\be   } dx 
	=C\tau^{-1/\beta}\,. 
\end{split}
\]
This proves the lemma. 
	\end{proof}
%
We need the following computations in the proof of next theorem.

To estimate $I_1$, we first  extend \eqref{e.2.9a}.  For any diagonal
matrix $D=\diag(d_1, d_2, \cdots, \cdots)$, we have 
\ba
\bigg( \sum_{i,j=1}^\infty  \al _{ij}    A_i^T D   A_j  \bigg)_{k,m}
&=&  \sum_{i,j,\ell=1}^\infty \al_{ij}     a_{i\ell k} d_\ell  a_{j\ell m}
\nonumber \\
&=& \sum_{i,j,\ell=1}^\infty   \int_{\cO^4} q(x,y) e_i(x) e_j(y) 
e_i(\xi) e_\ell (\xi) e_k(\xi) d_\ell e_j(\eta)   e_ \ell(\eta) e_m(\eta) dxdyd\xi d\eta	\nonumber \\ 
&=&    \sum_{ \ell=1}^\infty \int_{\cO^4} q(x,y) \de(x-\xi)   e_\ell (\xi)  e_\ell (\eta) \de(y-\eta) d_\ell 
e_k(\xi)     e_m(\eta) \d x\d y\d\xi \d\eta	\nonumber \\
&=&  \sum_{ \ell=1}^\infty \int_{\cO^2 } q(\xi,\eta)   d_ \ell e_\ell (\xi)
e_\ell (\eta) e_k(\xi)     e_m(\xi)  \d\xi \d \eta\,.  \label{e.3.18} 
\ea
We also have for any $M\ge 1$, 
\ba
\bigg( \sum_{i,j=M}^\infty  \al _{ij}    A_i^T D   A_j  \bigg)_{k,m}
&=&   \sum_{i,j=M} ^\infty \sum_{ \ell=1}^\infty    \al_{ij}     a_{i\ell k} d_\ell  a_{j\ell m}
\nonumber \\
&=& \sum_{i,j=M} ^\infty \sum_{ \ell=1}^\infty    \int_{\cO^4} q(x,y) e_i(x) e_j(y) 
e_i(\xi) e_\ell (\xi) e_k(\xi) d_\ell e_j(\eta)d_\ell  e_ \ell(\eta) e_m(\eta) dxdyd\xi d\eta	\nonumber \\
&=&  \sum_{ \ell=1}^\infty    \int_{\cO^2}  Q_M(  \xi,\eta)  e_\ell (\xi) e_k(\xi) d_\ell    e_ \ell(\eta) e_m(\eta)  d\xi d\eta\,, 	\label{e.3.18a}  \\ 
\ea
where 
\begin{equation} 
	Q_M( \xi, \eta)=\sum_{i,j=M}^\infty\int_{\cO^2} q(x,y) e_i(x)e_j(y) e_i(\xi)       e_j(\eta)  dxdy  \,. \label{e.def_Q} 
\end{equation}

To obtain the convergence of $U_{N,M}$ to $u$  we impose the following condition on  $\al_{ij}$. 
\begin{assumption}\label{assumption1} Let $\al_{ij}$ be defined by \eqref{e.2.2}. We assume 
	\begin{equation} 
		\begin{split}	
	\rho(M):=\sup_{\|f\|_2\le 1}		\int_{\cO^2}Q_M(\xi,\eta) f(\xi)f(\eta)\d\xi \d\eta
\to 0\quad \hbox{as $M\to \infty$}\,. 
		\end{split}
	\end{equation}  
\end{assumption}
\begin{theorem}\label{convergencetheorem} 
Let  the spatial dimension $d=1$ and let  $0<\gamma<1/2$ be given.  If  
\begin{equation}
	2(\la_1+\be_0) -\be_1^2\kappa>0\,,\
\end{equation}
then  there is a constant $C_T$ independent of $N$ and $M$ such that 
\[
\sup_{0\le t\le T}   	\||u(t)-U_{N,M}(t)\||_2^2\le C(\lambda_N^{- \gamma}+\rho(M))\,. 
\] 
\end{theorem}
\begin{proof}
By triangle inequality we can separate the error $\||u(t)-U_{N,M}(t)\||_2^2$ into two terms as follows
\begin{equation*}
	\||u(t)-U_{N,M}(t)\||_2^2\leq 2 \||u(t)-U_{N}(t)\||_2^2+2\||U_{N}(t)-U_{N,M}(t)\||_2^2.
\end{equation*}
In order to estimate $\||u(t)-U_{N}(t)\||_2^2$, we consider an infinite extension
\begin{equation}
	\d  \hat{U}_{N}(t)= -(\hat{\La}^N +\hat{B}^N)  \hat{U}_{N}(t) \d t +\beta_1 \sum_{j=1 }^\infty     \hat{A}^N_j \hat{U}_{N}(t)\d B_j(t),
	\label{e.3.16} 
\end{equation}
where  $\hat{\La}^N=\diag (\lambda_1, \cdots, \lambda_{N},0,\cdots)$, $\hat{B}^N=\diag (\beta_0, \cdots, \beta_0,0,\cdots)$, $\hat{A}^N_j=(\langle e_je_ie_k\rangle )_{1\le i,k\le N}$ and $\hat{A}^N_j= (0)_{i,\hbox{or}\ k\geq N}.$
Clearly, the equality $\||u(t)-U_{N}(t)\||_2^2=\||u(t)-\hat{U}_{N}(t)\||_2^2$ 
still  holds true.

In the following,  we denote $Z_N(t)=u(t)-\hat{U}_{N}(t)=(z_1(s), \cdots, z_i(s),\cdots)$.
Subtracting \eqref{e.3.16} from (\ref{e.2.6}) yields
\begin{equation*}
	\begin{split}
		\d  Z_N(t)= & -(\La+B) Z_{N}(t) \d t +\beta_1 \sum_{j=1 }^\infty     A _j Z_{N}(t)\d B_j(t)  \\
		&\quad -(\La  -\La_N )  \hat U_{N}(t) \d t-(B-\hat B_N) \hat U_{N}(t) \d t 
		+ \beta_1 \sum_{j=1 }^\infty     (A _j -\hat A_j^N)\hat U_{N}(t)\d B_j(t) \,.  
	\end{split} 
\end{equation*} 
Its integral form is 
\begin{equation*}
	\begin{split}
	 Z_N(t) =   & e^{-(\La+B)t} Z_N^0+\beta_1 \int_{0}^{t}e^{-(\La+B)(t-s)}\sum_{j=1}^\infty A_jZ_N(s) \d B_j(s) \\
		&+ \int_{0}^{t}e^{-(\La+B)(t-s)}(\La-\hat{\La}_N)\hat{U}_N(s) \d s  + \int_{0}^{t}e^{-(\La+B)(t-s)}(B-\hat B_N)\hat{U}_N(s) \d s \\
		&+ \beta_1\int_{0}^{t}e^{-(\La+B)(t-s)}\sum_{j=1}^\infty (A_j-\hat{A}_j^N)\hat{U}_N(s) \d B_j(s)\,,  
	\end{split}
\end{equation*}
where $Z_N^0$ is the initial value and $B=\diag(\beta_0,\cdots,\beta_0,\cdots)$.  Thus, we have 
\begin{equation}
	\begin{split}
		\| |Z_N(t)|\|_2^2= & 5^2\||e^{-(\La+B)t}Z_N^0|\|_2^2+5^2\beta_1^2 \||\int_{0}^{t}e^{-(\La+B)(t-s)}\sum_{j=1}^\infty A_jZ_N(s) \d B_j(s)|\|_2^2\\
		&+5^2\||\int_{0}^{t}e^{-(\La+B)(t-s)}(\La-\hat{\La}_N)\hat{U}_N(s) \d s|\|_2^2\\
		&+5^2\||\int_{0}^{t}e^{-(\La+B)(t-s)}(B-\hat B_N)\hat{U}_N(s) \d s|\|_2^2\\
		&+5^2\beta_1^2  \||\int_{0}^{t}e^{-(\La+B)(t-s)}\sum_{j=1}^\infty (A_j-\hat{A}_j^N)\hat{U}_N(s) \d B_j(s)|\|_2^2\\
		&\le  C (I_0+I_1+I_2+I_3+I_4)\,. 
		\label{Convergence1}
	\end{split}
\end{equation}

Next, we bound these five  terms separately.  If the initial values for    $u(t)$ and $U_{N,M}(t)$ are chosen to be the same,  then $I_0=0$.

The term $I_1$ can be bounded by utilizing the It\^o isometry,  and  the above identity
\eqref{e.3.18}, that is  
\begin{equation}
	\begin{split}
		I_1&=\EE (\int_0^t e^{-(\La+B)(t-s)}\sum_{j=1}^\infty A_j Z_N(s)\d B_j(s))^2\\
		& =   \EE\int_0^t(  Z_N(s))^T
		\sum_{m,j=1}^\infty \al_{mj} A_j  e^{-2(\La+B)(t-s)}A_m Z_N(s))\d s\\
		&=   	\sum_{m,j,\ell =1}^\infty  \int_{0}^t \int_{\cO^2 }
		q(\xi,\eta)  e^{-2(\la_ \ell+\be_0)(t-s)}  e_\ell (\xi)
		e_\ell (\eta) e_j(\xi)     e_m(\xi) z_m(s) z_j(s)  \d\xi \d \eta \d s \\
		&=   	\sum_ {\ell =1}^\infty  \int_{0}^t \int_{\cO^2 }
		q(\xi,\eta)  e^{-2(\la_ \ell+\be_0)(t-s)}  e_\ell (\xi)
		e_\ell (\eta)  Z_N(s,\xi) Z_N(s,\eta) \d\xi \d \eta \d s \,. 
	\end{split}
\end{equation} 
Now  by   H\"older's inequality,  $\sup_{x,y \in \cO  }
|q(x,y)|\le C$,  and $\sup_{ x\in \cO} |e_k(x)|\le C$,  we have  
\ba
I_1
&\leq& C\EE \int_{0}^t (\int_{\cO^2} q^2(x,y)Z_N^2(s,x)Z_N^2(s,y)\d x \d y)^{\frac{1}{2}}
\nonumber \\
&&\qquad \int_{0}^t (\int_{\cO^2}(\sum_{k =1}^\infty e^{-2(\lambda_k+\beta_0)(t-s)}e_k(x)e_k(y))^2 \d x\d y)^{\frac{1}{2}} \d s
\nonumber  \\
&\leq&  C \int_{0}^t \||Z_N(s)|\|_2^2 (\sum_{k =1}^\infty e^{-4(\lambda_k+\beta_0)(t-s)} )^{\frac{1}{2}} \d s \nonumber \\
&\leq&  C \int_{0}^t \||Z_N(s)|\|_2^2 (\sum_{k =1}^\infty e^{-4c_0k^2(t-s)} )^{\frac{1}{2}} \d s \nonumber \\
&\le& C\int_0^t (t-s)^{-\frac{1}{4}}\||Z_N(s)|\|_2^2\d s \,, 
\ea
since $\la_k\asymp c_0k^2$ as $k\to \infty$ and  where the last inequality follows from \eqref{e.3.17}  with $\beta=2$. 
Before  we   proceed  with the estimate of $I_2$, we need to obtain   some more bounds  for  the solution.  For any positive $k\ge 1$, 
using  (\ref{e.3.2}),   Theorem \ref{exactstability}  and the Burkholder-Davis-Gundy inequality we see 
\ba 
I_{2,k} &:= & \EE |(\lambda_k+\beta_0)^{\frac{1}{2}}u_{N, k}(t)|^q
\nonumber \\
&\leq&  C\EE |(\lambda_k+\beta_0)^{\frac{1}{2}}e^{-(\lambda_k+\beta_0)t}u_{N, k}(0)|^q\nonumber  \\
&&\qquad +C \EE |\sum_{j =1}^\infty \int_0^t (\lambda_k+\beta_0)^{\frac{1}{2}}e^{-(\lambda_k+\beta_0)(t-s)}\sum_{i=1}^\infty   a_{ jik }u_{N, i}(s)\d B_j(s)|^q\nonumber  \\
&	\leq&  C + C \EE(\int_{0}^t (\lambda_k+\beta_0)e^{-2(\lambda_k+\beta_0)(t-s)}\nonumber\\
&&\qquad\times \sum_{m,j=1}^\infty \al_{mj}(\sum_{i=1}^ \infty a_{jik }u_{N, i}(s))(\sum_{i=1} ^\infty a_{mi'k}u_{N, i}(s))\d s)^{\frac{q}{2}}\nonumber\\
&=&  C + C \EE(\int_{0}^t (\lambda_k+\beta_0)e^{-2(\lambda_k+\beta_0)(t-s)}
\tilde I_{2,k}(s)ds)^{q/2}   \,,  \nonumber 
\ea
where we have as in \eqref{e.alA} 
\bas 
\tilde I_{2,k}(s)
&=&   \sum_{i, i'=1}^\infty  \sum_{m,j=1}^\infty \al_{mj} a_{jik }a_{mi'k} u_{N, i}(s)  u_{N, i'}(s) \\ 
&=&   \sum_{i, i'=1}^ \infty  \int_{\cO^2} 
q(\xi, \eta) e_i(\xi)e_k(\xi )
e_{i'}(\eta) e_k(\eta )d\xi d\eta \,  
 u_{N, i}(s)  u_{N, i'}(s)\\
&=&     \int_{\cO^2} 
q(\xi, \eta) e_k(\xi )
  e_k(\eta )d\xi d\eta \,  
u_{N }(s,\xi)  u_{N }(s, \eta)\,. 
\eas 
Therefore, we have 
\ba 
I_{2,k}  
&\leq & C+C\EE(\int_{0}^t (\lambda_k+\beta_0)e^{-2(\lambda_k+\beta_0)(t-s)} \int_{\cO^2}q(x,y)e_k(x)e_k(y)u(t,x)u(t,y)\d x \d y \d s)^{\frac{q}{2}}
\nonumber  \\
&\leq &  C+  C (\int_{0}^t (\lambda_k+\beta_0)e^{-2(\lambda_k+\beta_0)(t-s)}(\int_{\cO^2}e_k^2(x)e_k^2(y)\d x\d y)^{\frac{1}{2}}  \nonumber\\
&&\qquad\quad \EE (\int_{\cO^2}q^2(x,y)\hat u_N^2(s,x)\hat u_N^2(s,y)\d x \d y)^{\frac{1}{2}} \d s)^{\frac{q}{2}}\nonumber  \\
&\leq&  C+C (\int_{0}^t (\lambda_k+\beta_0)e^{-2(\lambda_k+\beta_0)(t-s)} \d s)^{\frac{q}{2}}  \leq  C\,,    
\label{I22}  
\ea
where  in the  last inequality we used the fact that $\sup_{x,y \in \cO  }
|q(x,y)|\le C$ and $\sup_{t,x\in \RR_+\times \cO}\EE |U_N(t,x)|^q\le C$.  
For any conjugate numbers 
$1/p+1/q=1$, we have by H\"older's inequality 
\ba
I_2&=&\EE \sum_{k=N+1}^\infty (\int_{0}^t e^{-(\lambda_k+\beta_0)(t-s)}(-(\lambda_k+\beta_0))u_{N, k}(s)\d s)^2\nonumber \\
&\leq& \sum_{k=N+1}^\infty (\int_{0}^t e^{-p(\lambda_k+\beta_0)(t-s)}(\lambda_k+\beta_0)^{\frac{p}{2}}\d s)^{\frac{2}{p}}\EE(\int_{0}^t (\lambda_k+\beta_0)^{\frac{q}{2}}|u_{N, k}(s)|^q\d s)^{\frac{2}{q}} \nonumber \\
&\leq&  \sum_{k=N+1}^\infty (\lambda_k+\beta_0)^{1-\frac{2}{p}}(\int_{0}^t\EE (\lambda_k+\beta_0)^{\frac{q}{2}}|u_{N, k}(s)|^q\d s)^{\frac{2}{q}}\nonumber \\
&:=& \sum_{k=N+1}^\infty ((\lambda_k+\beta_0)^{1-\frac{2}{p}}(\int_{0}^t I_{2,k} \d s)^{\frac{2}{q}}.
\label{I2} 
\ea
By \eqref{I22},   we have 
\begin{equation}\label{II2}
	I_2\leq C \sum_{k=N+1}^\infty (\lambda_k+\beta_0)^{1-\frac{2}{p}}.
\end{equation}

$I_3$ can be done with similar    (in fact  easier) way as  for $I_2$.
It can be bounded by something better   but   at least can be 
bounded by \eqref{II2}. We have
\begin{equation}\label{II3}
	I_3\leq C \sum_{k=N+1}^\infty (\lambda_k+\beta_0)^{1-\frac{2}{p}}.
\end{equation}

In the sequel, we turn to  the estimate of $I_4$.   By similar estimation as $I_1$,  with the aid of the   Burkholder-Davis-Gundy-type inequality and Theorem \ref{theoremsemi1}, we get
\begin{align*}
	I_4&=\EE (\int_{0}^{t}e^{-(\La+B)(t-s)}\sum_{j=1}^\infty (A_j-\hat{A}_j^N) {U}_N(s) \d B_j(s))^2\\
	&\leq C\EE \int_{0}^{t}\sum_{k=N+1}^{\infty}e^{-2(\lambda_k+\beta_0)(t-s)}
	\sum_{m,j=1}^\infty(\sum_{i=1}^\infty a_{kji}u_{N, i}(s) )(\sum_{i=1}^\infty a_{kmi}u_{N, i}(s)) \al_{mj} \d s\\
	& \leq  C\int_{0}^{t} \sum_{k=N+1}^{\infty} e^{-(\lambda_{k}+\beta_0)(t-s)}\||  {U}_N(s)|\|_2^2 \d  s\leq C \sum_{k=N+1}^{\infty} (\lambda_k+\beta_0)^{-1}.
\end{align*}
Since $p>1$ we see  $1-2/p\ge -1$.  It follows from the above  (\ref{Convergence1})-(\ref{II2})    that 
\begin{equation}
	\begin{split}
		\| |Z_N(t)|\|_2^2 & \leq C \sum_{k=N+1}^{\infty} (\lambda_k+\beta_0)^{1-\frac{2}{p}}+C\int_0^t (t-s)^{-\frac{1}{4}}\||Z_N(s)|\|_2^2\d s\,.  
	\end{split}\label{e.3.30}
\end{equation} 
Thus, by  Lemma \ref*{lemma}  we obtain 
\begin{equation}
	\begin{split}
		\| |Z_N(t)|\|_2^2 
		&\leq C \sum_{k=N+1}^{\infty} (\lambda_k+\beta_0)^{1-\frac{2}{p}} E_{ \frac{3}{4},1}(\Ga( \frac{3}{4})^{\frac{1}{2}}t)\\
		&\leq C \sum_{k=N+1}^{\infty} (\lambda_k+\beta_0)^{1-\frac{2}{p}}\\
		&\leq  C \sum_{k=N+1}^{\infty} (\lambda_k+\beta_0)^{-1+\gamma} (\lambda_k+\beta_0)^{-\frac{2}{p}-\gamma+2}\\
		&\leq C (\lambda_{N+1}+\beta_0)^{-1+\gamma} \leq C \lambda_{N}^{-1+\gamma}\,, 
	\end{split}
\end{equation} 
where   $\frac{2}{p}+\gamma>5/2$, $0<\gamma<1$ and where in the 
above second last inequality, we used the fact that $\la_k\sim k^2$ as 
$k\to \infty$ which implies that $\sum_{k=N+1}^{\infty}   (\lambda_k+\beta_0)^{-\frac{2}{p}-\gamma+2}<\infty$.  Now we look at the condition
for $\gamma$.  For any $\gamma>1/2$, we can choose $p>1$ 
such that	 $\frac{2}{p}+\gamma>5/2$. Thus, we have for any $\ga_0=1-\ga<1/2$, we have
\begin{equation} 
	\| |Z_N(t)|\|_2^2   \leq C \lambda_{N}^{-\gamma_0}\,. 
\end{equation}

We are now ready to  estimate $\||U_N(t)-U_{N,M}(t)\||_2^2$.  
Denote  $K_{N,M}(t)=\hat{U}_{N}(t)-\hat{U}_{N,M}(t)$.
Subtracting \eqref{HeatSemimatrix1}  from \eqref{HeatSemimatrix} yields 
\ba
\| |K_{N,M}(t)|\|_2^2
&= & 3^2\||e^{-(\hat{\La}^N +\hat{B}^N)t}K_{N,M}^0|\|_2^2+3^2\be_1^2\||\int_{0}^{t}e^{-(\hat{\La}^N +\hat{B}^N)(t-s)}\sum_{j=1}^\infty \hat{A}_j^NK_{N,M}(s) \d B_j(s)|\|_2^2
\nonumber \\
&&\qquad +3^2\be_1^2 \||\int_{0}^{t}e^{-(\hat{\La}^N +\hat{B}^N)(t-s)}\sum_{j=M+1}^\infty \hat{A}_j^N\hat{U}_{N,M}(s) \d B_j(s)|\|_2^2 	\nonumber \\
& \le  & C( I_5+I_6+I_7),
\label{Convergence2}
\ea
where $K_N^0$ is the initial value. 
Again,  we choose the same initial values for   $\hat{U}_{N}(t)$ and $\hat{U}_{N,M}(t)$ so that  $I_5=0$.
The term $I_6$ can be similarly bounded with $I_1$, that is
\begin{equation*}
	I_6 \leq C\int_0^t (t-s)^{-\frac{1}{4}}\||K_{N,M}(s)|\|_2^2\d s.
\end{equation*}

Regard the term $I_7$,  by the computation  
\eqref{e.3.18a},  under the Assumption \ref{assumption1} and Inequality \eqref{e.3.15} of Theorem \ref{Theoremfinite}, we derive 
\begin{equation}
		\begin{split}
			I_7&\leq  C\int_{0}^t \sum_{j,m=M+1}^\infty \al_{mj} (e^{-(\hat{\La}^N +\hat{B}^N)(t-s)}\hat{A}_j^N \hat{U}_{N,M}(s))^T(e^{-(\hat{\La}^N +\hat{B}^N)(t-s)}\hat{A}_m^N \hat{U}_{N,M}(s)) \d s \\
			&\leq C\int_{0}^t \sum_{k=1}^N \sum_{i ,j=1}^\infty  e^{-2(\lambda_k+\beta_0)(t-s)}    \int_{\cO^2} Q_M(x,y) e_k(x)e_i(x)e_k(y)e_{j} (y)\hat{u}_{N,M}(s,x)\hat{u}_{N,M}(s,y)\d x\d y\d s\\
			&\leq C\int_{0}^t \sum_{k=1}^N   e^{-2(\lambda_k+\beta_0)(t-s)}    \int_{\cO^2} Q_M(x,y) e_k(x) e_k(y) \hat{u}_{N,M}(s,x)\hat{u}_{N,M}(s,y)\d x\d y\d s\\
				&\leq C  \rho(M) \sup_{ k\ge 1\,,  x\in \cO} |u_k(x)|^2 \int_{0}^t \sum_{k=1}^N   e^{-2(\lambda_k+\beta_0)(t-s)}     \|\hat{u}_{N,M}(s )\|_2^2   \d s\le C \rho(M) 
		\end{split}
	\end{equation} 
where the last inequality follows from Theorem \ref{Theoremfinite}.   This gives 
\begin{equation*}
	\||K_{N,M}(t)|\|_2^2 \leq C\int_0^t (t-s)^{-\frac{1}{4}}\||K_{N,M}(s)|\|_2^2\d s
	+  C \rho(M).
\end{equation*}
By  utilizing Lemma \ref{lemma} as   before  we have
\begin{equation}
	\begin{split}
		\| |K_{N,M}(t)|\|_2^2 
		&\leq C \rho(M) E_{ \frac{3}{4},1}(\Ga( \frac{3}{4})^{\frac{3}{4}}t) \leq C \rho(M).
	\end{split}
\end{equation} 
Finally,  gathering the above estimates gives
\[
\sup_{0\le t\le T}   	\||u(t)-U_{N,M}(t)\||_2^2\le C(\lambda_N^{-1+\gamma}+\rho(M))\,. 
\] 
This completes the proof of the theorem. 
\end{proof}
\begin{example}
If we choose the  domain $\cO$ as  $[0,1]$, then the eigenvalues and eigenvectors of the Laplace operator $-\Delta$ with the Dirichlet boundary condition are $\lambda_k=k^2 \pi^2$, $e_k(x)=\sqrt{\frac{2}{\pi}}\sin k\pi x$. Particularly,  if    the   eigenvectors of   the covariance $q(x,y)$ is also $\{e_k(x), k=1, 2, \cdots\}$,  and if $q(x,y)$ has   continuous differentiable derivatives up to order $2\ell$ with zero values on the boundary of $[0, 1]^2$,  then  
integration by parts $2\ell$ times gives (we denote 
$\tilde e_i(x)=e_i(x)$   if $\ell $ is even and 
$\tilde e_i(x)=\sqrt{\frac2\pi} \cos(k\pi x)$ if   $\ell $ is odd)
\begin{equation}
	\begin{split}
	\rho(M) &\le C  \sum_{i,j=M}^\infty \left| \int_{[0, 1]^2} 
	q(x,y) e_i(x)e_j(y) dxdy\right|\\ 
&\le C  \sum_{i,j=M}^\infty {i^{-\ell} j^{-\ell} }\left| \int_{[0, 1]^2} 
\frac{\partial ^{2\ell }}{\partial x^\ell 
\partial y^\ell }q(x,y) \tilde e_i(x)\tilde e_j(y) dxdy\right|\\ 
&\le C  \sum_{i,j=M}^\infty {i^{-\ell} j^{-\ell} } 
\le \int_{M}^\infty \int_{M}^\infty (xy)x^{-\ell}dxdy=CM^{-2\ell +2}
\,. 
	\end{split}
\end{equation}
Thus, we have in this case 
\[
\sup_{0\le t\le T}   	\||u(t)-U_{N,M}(t)\||_2^2\le C(N^{-2+2\gamma}+M^{-2\ell +2})\,.  
\]   
\end{example}

\section{Full discretization} \label{s.full}
We consider the uniform partition in time. Let
$t_j=j \tau $,  where $\tau$ is the stepsize. Applying the implicit Euler method to (\ref{HeatSemimatrix1}) leads to
\begin{equation}\label{HeatFull}
U_{n+1}= U_n - \tau \La^N U_{n+1} - \tau B^N U_{n+1}+\beta_1 \sum_{j=1}^M  A^N_j U_n\Delta B_j^n,
\end{equation}
where we follow the notation  of (\ref{HeatSemimatrix1}),  $U_0=P_N u_0$,   $\Delta B_j^n= B_j(t_{n+1})-B_j(t_n)$. 
$U_n$ is an $\RR^N$-valued random variable and is the approximation of   solution $U(t_n)$.
To study the stability of the random sequence of $U_n$ as $n$ goes to infinity
we define 
\[
\|U_n\|_2^2 =U_n^TU_n
\]
and
\[
\|A\|_2=\sum_{x\in \RR^N\,, \|x\|_2=1}\|Ax\|_2
\]
for a matrix from $\RR^N$ to $\RR^N$ or from $\RR^\infty$ to $\RR^\infty$. 

Our next result is 
\begin{theorem}\label{Heattheorem}
Under the condition (\ref{e.2.12}), the numerical solutions generated by (\ref{HeatFull}) are mean-square stable in the  norm sense  for any stepsize $\tau$.
\end{theorem}
\begin{proof}
 Considering the equivalent form of (\ref{HeatFull}), we have
	\begin{equation}\label{HeatFull}
		U_{n+1}= U_n - \tau \La^N U_{n+1} - \tau B^N U_{n+1}+\beta_1 \sum_{j=1}^\infty  A^N_j U_n\Delta  \cB_{j}^N.
	\end{equation}
Taking the expectation on above equation
\begin{equation*}
	\begin{split}
		\EE  \left[U_{n+1}^T (I+\tau(\La^N+B^N))^2U_{n+1}\right]
		=& \EE \left[ U_n^TU_n\right] + \tau \be_1^2 \sum_{i,j=1}^\infty \hat{\al}_{M,ij} \EE \left[  U_n^T(A^N_i)^T A^N_j U_n\right].
	\end{split}
\end{equation*} 
Under the condition (\ref{e.2.12}), $(I+\tau(\La^N+B^N))^2$ is a symmetric positive definite  matrix, then
\begin{equation*}
	\EE \left[ U_{n+1}^T (I+\tau(\La^N+B^N))^2 U_{n+1}\right] \geq (1+\tau(\lambda_1+\be_0))^2\EE \left[ U_{n+1}^T  U_{n+1}\right].
\end{equation*} 
  Similar to the proof of Theorem \ref{Theoremfinite}, we have 
\begin{equation*}
	\begin{split}
		\sum_{i,j=1}^\infty \hat{\al}_{M,ij}  \left[  U_{n}^T (A_i^N)^T   A_j^N U_{n}\right]  
		=& \sum_{k, m,\ell=1}^N  \int_{\cO^2} \hat{q}_M(\xi,\eta)   u_{N,M,k}(t_n)
		e_k(\xi)     u_{N,M,m}(t_n) e_m(\eta) e_\ell (\xi)  e_ \ell(\eta) \d\xi \d\eta\\
		=&     \int_{\cO^2} \hat{q}_{M,N}(\xi,\eta)   \left(\sum_{k  =1}^{ {\infty}} \hat{u}_{N,M,k}(t_n)
		e_k(\xi) \right)   \left(\sum_{  m =1}^{ {\infty}}  \hat{u}_{N,M,m}(t_n) e_m(\eta)\right)    \d\xi \d\eta  \\
		=&  \int_{\cO} \hat{q}_{M,N} (\xi,\eta)  \hat{U}_{N,M}(t_n, \xi)  \hat{U}_{N,M}(t_n, \eta)  \d\xi \delta\eta \\
		\le& \tilde \kappa_2   \int_{\cO}    (\hat{U}_{N,M}(t_n, \xi) ) ^2   \d\xi  \\
		=& \tilde \kappa_2  \||U_n\||_2^2.
	\end{split}
\end{equation*}
	As  a consequence,
	\begin{equation*}
		\||U_{n+1}\||_2^2 \le \frac{1+\tau \tilde \kappa_2 \be_1^2 }{(1+\tau(\lambda_1+\be_0))^2} \||U_n\||_2^2.
	\end{equation*}
	Thus, the numerical solutions are mean-square stable with any stepsize under the condition  (\ref{e.2.12}). The proof is hence completed. 
\end{proof}

\begin{remark}
Under  the	Neumann boundary condition, the numerical stability condition in Theorem  \ref{Heattheorem} is replaced by $	2\be_0 -\be_1^2\kappa>0$.
\end{remark}

\subsection{Other numerical method}
Applying the explicit Euler method to (\ref{HeatSemimatrix1}) leads to
\begin{equation}\label{Heattheta}
U_{n+1}= U_n  - \tau (\La^N+B^N )U_{n}+\sum_{j=1}^M  A^N_j U_n\Delta B_j^n.
\end{equation}
  \begin{theorem}
	The numerical solutions produced by (\ref{Heattheta}) are mean-square stable in the  norm sense if 
	\begin{equation}\label{Heatstabilitycondition}
		\lambda_{max}\{(I-\tau (\La^N+B^N))^2\}+\tau \tilde \kappa_2\beta_1^2 <1.
	\end{equation}
\end{theorem}
Clearly, the stability condition of the explicit Euler method is to restricted, since it is difficult to choose suitable coefficients to satisfy the condition, even for  given sufficient small coefficients, stepsize restrictions have to be imposed for the stiffness and much computational costs will be paid. Thus, a implicit scheme  for stiffness is  important.

Applying the stiff implicit-Euler method to (\ref{HeatSemimatrix1}) leads to
\begin{equation}\label{Heatstiff}
	U_{n+1}= U_n - \tau \La^N U_{n+1} - \tau B^N U_{n}+\beta_1 \sum_{j=1}^M  A^N_j U_n\Delta B_j^n.
\end{equation}
\begin{theorem}
	Under the condition (\ref{e.2.12}), the numerical solutions produced by (\ref{Heatstiff}) are mean-square stable in the spectral norm sense.
\end{theorem} 

\section{Numerical experiments}\label{s.5} 
In this section, some numerical experiments are performed to visually illustrate the previously claimed
numerical mean-square stability. For simplicity  we take
below 
\[{W}(t,x)=\sum_{j=1}^{\infty}\sqrt{q_j}e_j(x) {B}_j(t),\]
in which ${B}_j(t), j \in \mathbb{N}$ is a sequence of real-valued standard Brownian motions that are
mutually independent on the probability space $(\Omega, \mathcal{F}, \mathbb{P})$ and ${e_j = \sqrt{2} \sin (j\pi x)}$ are an orthonormal basis. Obviously, by a simple calculation, the correlation is given as follows
\[q(x,y)=\sum_{j=1 }^\infty q_j e_j(x) e_j(y). \]
\begin{example}
We consider the following equation
\begin{equation}\label{Heatexampleequation}
	\frac{	\partial u}{\partial  t} (t,x) = \frac{1}{2} \nu \Delta u(t,x)  +\sqrt{\lambda} u(t,x) \dot W(t,x),   ~~~~~x\in [0,1],
\end{equation}
where $\nu, \lambda>0 $ are physical parameter.
The analogue  of equation (\ref{Heatexampleequation}) has been proposed and analyzed in detail as a linearized model for a diffusion in non stationary random media (see \cite{Bertini1999,Carmona1994,Molchanov1991}). It is easy to see that the stability condition of equation (\ref{Heatexampleequation}) through  Theorem  \ref{Heattheorem} is 
\begin{align*}
	-  \nu \lambda_1+\lambda \sum_{j=1}^\infty q_j<0.
\end{align*}
\end{example}
The curves on the left of Figure  \ref{fig:1} are plotted with  $\lambda=\frac{1}{4}, q_j=j^{-1.001}, j=1,\cdots,10$     and  on the right of Figure \ref{fig:1} are plotted $\nu=2, q_j=j^{-1.001}, j=1,\cdots,10$. Figure \ref{fig:1} gives   the impact of the parameters $\nu$ and $\lambda$. Fix  other parameters  and vary only  $\nu$ and $\lambda$, respectively. It is shown that the curve decrease to zero along the time. Moreover, the curves decrease faster for a smaller $\nu$ and larger $\lambda$.
\begin{figure}[!tbp]
\centering
\subfigure[ $\lambda=\frac{1}{4}$ ]{\includegraphics[width=6cm]{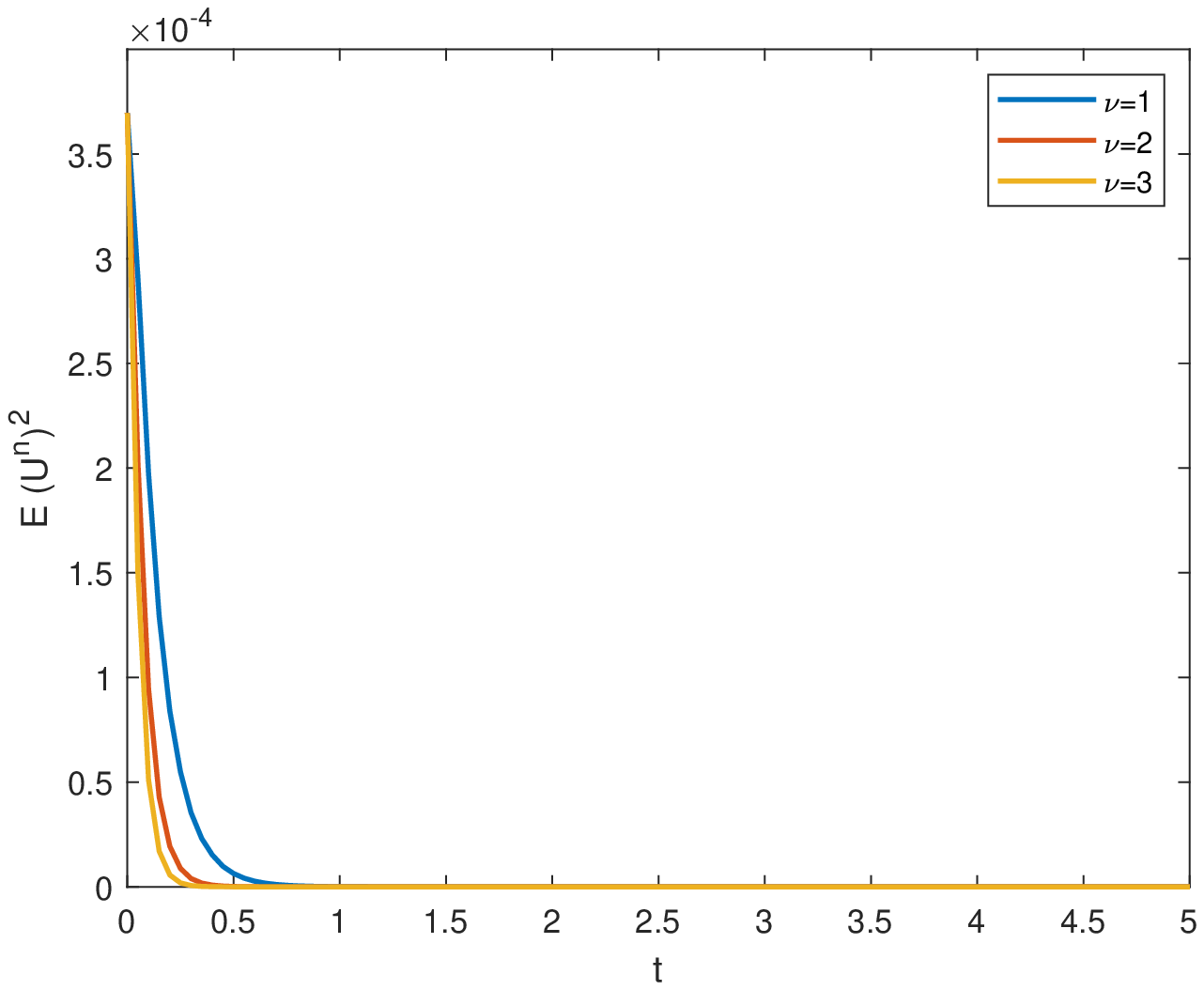}}
\subfigure[ $\nu=2$ ]{\includegraphics[width=6cm]{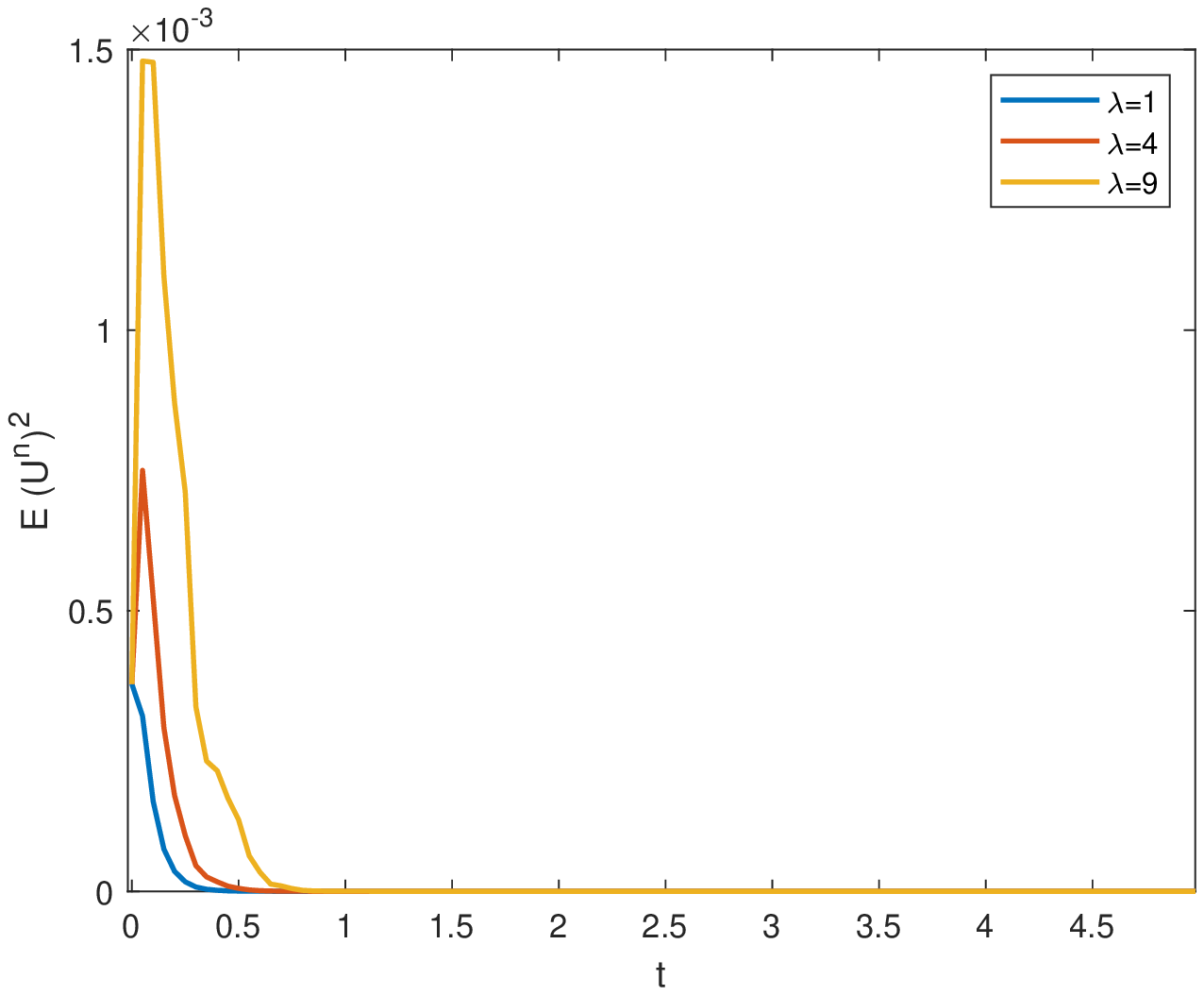}}
\caption{The mean-square curves of equation (\ref{Heatexampleequation}).}
\label{fig:1}
\end{figure}

\begin{example}
We are focus the following test equation
\begin{equation}
	\frac{	\partial u}{\partial  t} (t,x) = \Delta u(t,x) + \beta_0 u(t,x) +\beta_1 u(t,x) \dot W(t,x),   ~~~~~x\in [0,1],
\end{equation}
with the initial condition  $u(0,x)=x(1-x)$.
\end{example}
Since the mean-square stability region of the implicit Euler method depends upon $\beta_0, \beta_1^2\sum\limits_{j}q_j$, by analogy with the standard practice for  stability regions \cite{Higham2000},  we fix $q_j$ and draw these regions in the $x-y$ plane, where $x=\beta_1$ and $y=\beta_0$.

The left-hand and the right-hand pictures in Figure  \ref{figstabilityregion} illustrate the cases $q_j=j^{-1.001}$ and $q_j=j^{-2.001}$, respectively. The analytical mean-square stability region is shown with red color.  Surrounded by yellow solid line is the region where the numerical solutions are stable.
\begin{figure}[!tbp]
\centering
\subfigure[ $q_j=j^{-1.001}$ ]{\includegraphics[width=6cm]{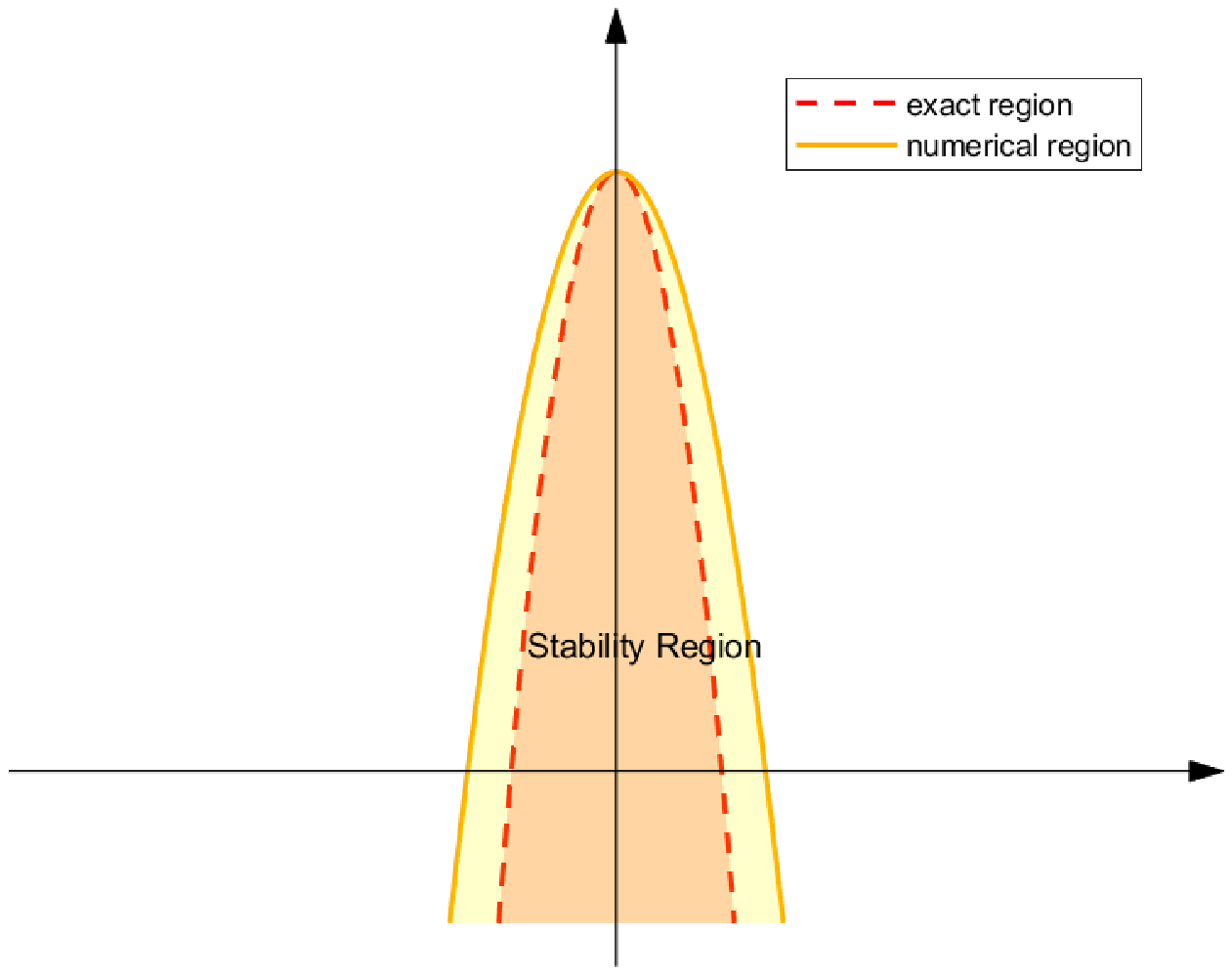}}
\subfigure[$q_j=j^{-3.001}$  ]{\includegraphics[width=6cm]{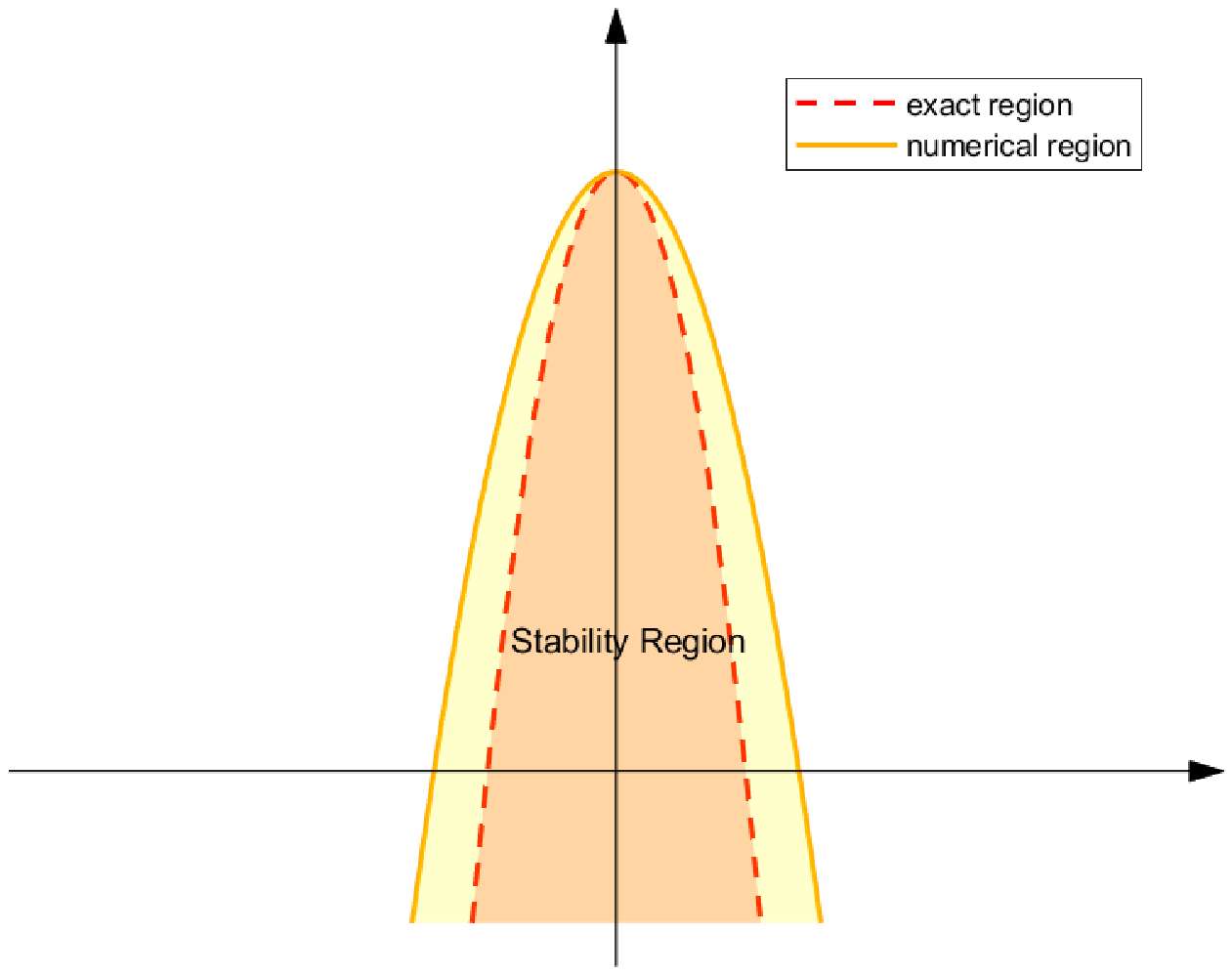}}
\caption{Mean-square stability region of the numerical solutions and the analytical solutions.}
\label{figstabilityregion}
\end{figure}

We choose the parameter with $\beta_0=1, \beta_1=1$ and variable $q_j$ for $j=1,\cdots,100$. It is shown that the curves decrease to zero in Figure \ref{figstability2}. 
\begin{figure}[!tbp]
\centering
\subfigure[ $q_j=j^{-1.001}$ ]{\includegraphics[width=6cm]{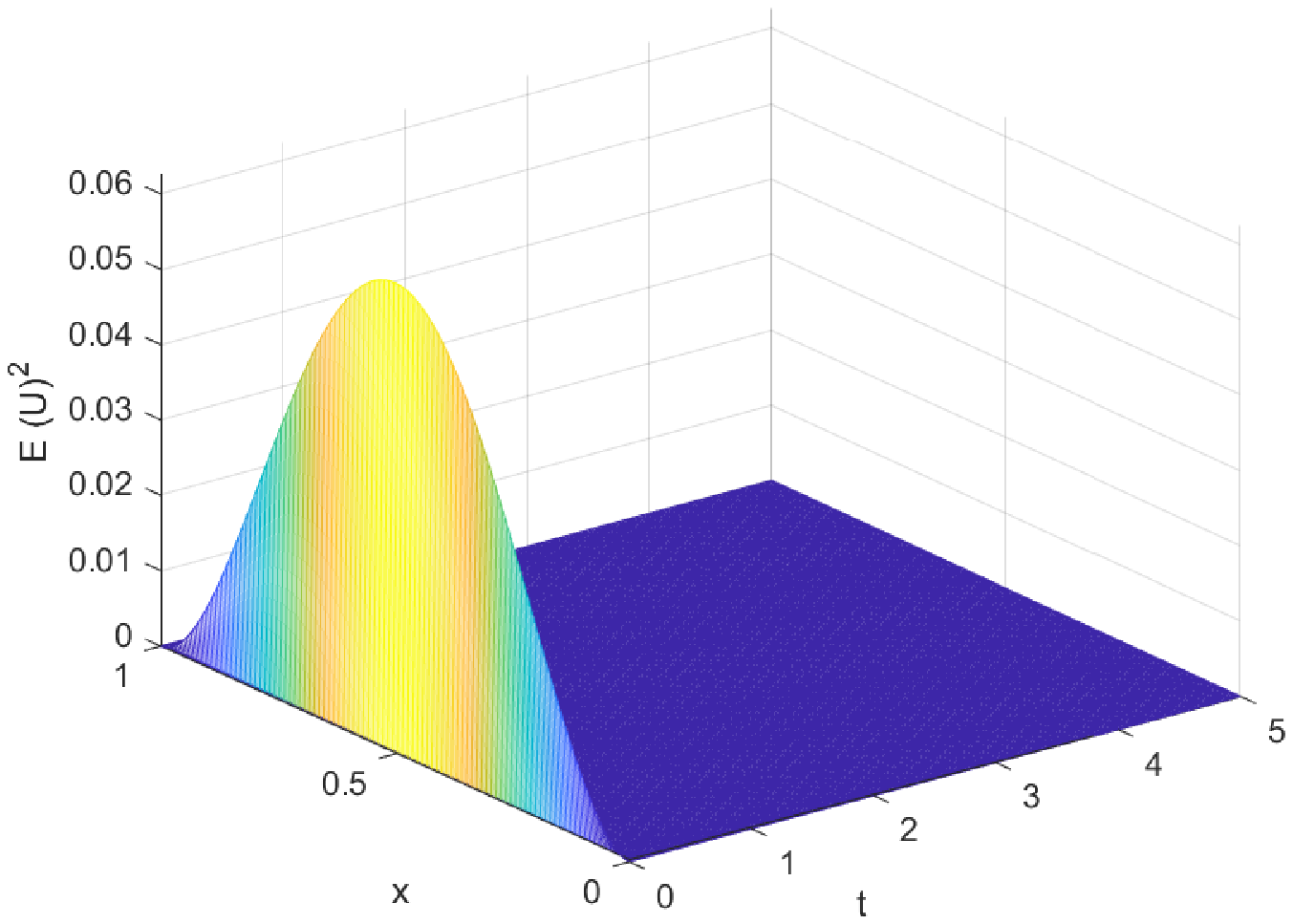}}
\subfigure[$q_j=j^{-3.001}$  ]{\includegraphics[width=6cm]{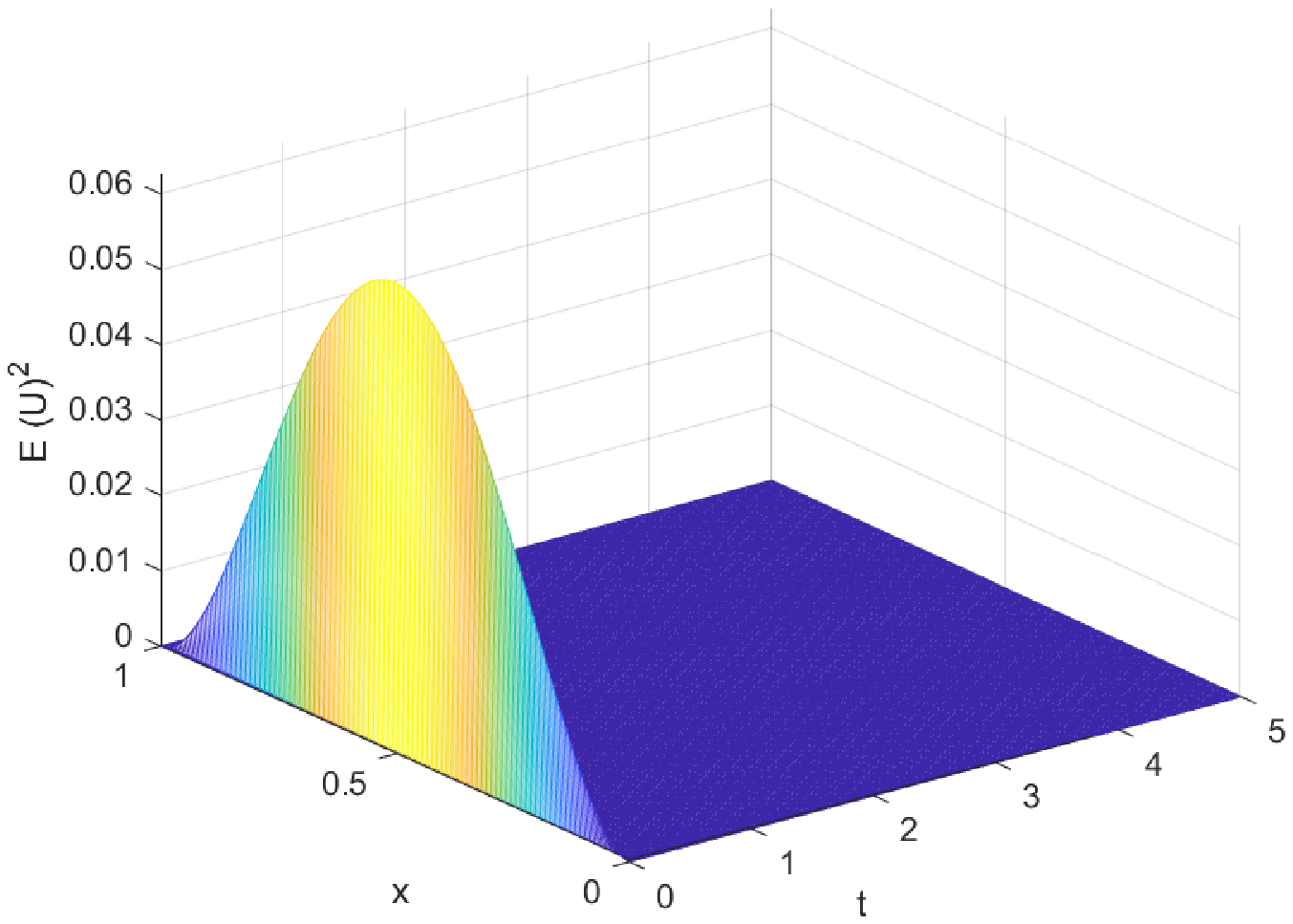}}
\caption{Mean-square curve of numerical solutions with $\beta_0=1, \beta_1=1$.}
\label{figstability2}
\end{figure}
In addition, we are also interested in the impact of different numerical method. We choose $\beta_1=1, q_j=j^{-1.001}, j=1,\cdots,10$.  Figure \ref{figstability3} shows  that the curves  tend to zero along the time. In the upper of Figure \ref{figstability3}, there is no sharp distinction for a smaller $\beta_0$. For a larger $\beta_0$, it is easily see from the lower of  Figure \ref{figstability3}  that a fast downward trend  with implicit Euler method. 
\begin{figure}[!tbp]
\centering
\subfigure[ stiff-implicit Euler method with $\beta_0=0.2$ ]{\includegraphics[width=6cm]{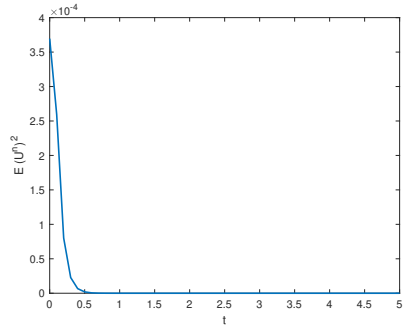}}
\subfigure[implicit Euler method  $\beta_0=0.2$ ]{\includegraphics[width=6cm]{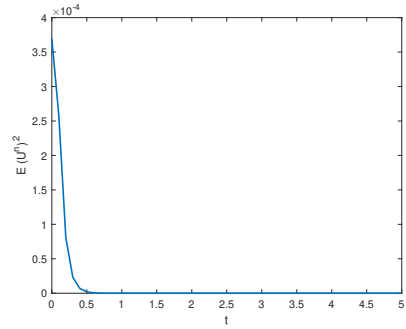}}\\
\subfigure[ stiff-implicit Euler method  $\beta_0=5$ ]{\includegraphics[width=6cm]{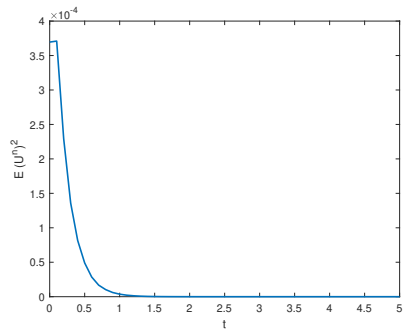}}
\subfigure[implicit Euler method  $\beta_0=5$ ]{\includegraphics[width=6cm]{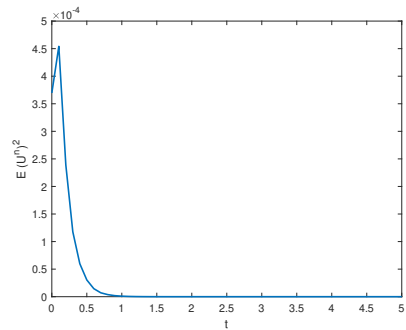}}
\caption{Mean-square curve of numerical solutions with different numerical method.}
\label{figstability3}
\end{figure}

\section[]{Conclusions}
In this article, we   investigate  the long time stability  of  stochastic heat equations  driven by a correlated noise, including the exact solution and its approximations.  By  the eigenfunction expansion method, we transform the SPDE to a system of infinitely many stochastic differential equations and obtain  the  condition on the covariance function for   the exact solutions
to be long time stable. This condition is much better than the existing one and    is also much      easier   to verify.  Furthermore, we substantiate our claim that the solution can be effectively approximated using its finite-dimensional spectral approximation, which, notably, maintains long-time  stability.
It should be mentioned here that we have considered similar problems with a nonlinear stochastic heat equation. But the  extension   to the nonlinear case  is still ongoing.

\section*{Acknowledgment}
Y. Hu is supported by an NSERC Discovery grant   RGPIN-2018-05687   and a centennial  fund from University of Alberta.

\end{document}